\documentclass{article}

\usepackage{arxiv}

\usepackage[utf8]{inputenc} 
\usepackage[T1]{fontenc}    
\usepackage{hyperref}       
\usepackage{url}            
\usepackage{booktabs}       
\usepackage{amsfonts}       
\usepackage{nicefrac}       
\usepackage{microtype}      
\usepackage{lipsum}		
\usepackage{graphicx}
\usepackage{algorithmic}
\usepackage{doi}
\usepackage{subcaption}
\usepackage[section]{placeins}

\usepackage{textalpha}
\usepackage{amssymb}

\usepackage{longtable}
\usepackage[ruled,vlined]{algorithm2e}
\usepackage{pgfplots}
\pgfplotsset{compat=newest}
\usepgfplotslibrary{fillbetween}
\usepackage{textalpha}
\usepackage{xcolor}

\usepackage[T1]{fontenc} 

\usepackage[utf8]{inputenc} 

\usepackage{graphicx} 
\graphicspath{{Figures/}} 

\usepackage{enumitem} 

\usepackage{lipsum} 

\usepackage{amsmath,amsthm} 

\usepackage{varioref} 

\usepackage{tikz}
\usepackage{tikz-cd} 
\usetikzlibrary{arrows,shapes,positioning,calc,3d}
\usepackage{longtable}
\usepackage{pgfplots}
\pgfplotsset{compat=newest}
\usepgfplotslibrary{fillbetween}
\usepackage{placeins}
\theoremstyle{definition} 

\theoremstyle{plain} 
\newtheorem{theorem}{Theorem}

\theoremstyle{remark} 
\newtheorem*{remark}{Remark}

\newcommand{\R}{\mathbb{R}}
\newcommand{\usol}{u}

\newcommand{\x}{x}

\newcommand{\eps}{\epsilon}
\newcommand{\Omeps}{\Omega_{\eps}}
\newcommand{\Om}{\Omega}

\newcommand{\Pv}{\vec{P}}
\newcommand{\nv}{\vec{n}}
\newcommand{\Vv}{\vec{V}}

\newcommand{\Qv}{\vec{Q}}
\newcommand{\rv}{\vec{R}}
\newcommand{\nab}{\nabla}

\newcommand{\Ga}{\Gamma_1}
\newcommand{\Gb}{\Gamma_2}

\newcommand{\Gd}{\Gamma_2}
\newcommand{\Ge}{\Gamma_3}
\newcommand{\Lag}{\mathcal{L}}

\newcommand*\diff{\mathop{}\!\mathrm{d}}

\pgfdeclarelayer{ft}
\pgfdeclarelayer{bg}
\pgfsetlayers{bg,main,ft}
\newcommand\restr[2]{{
		\left.\kern-\nulldelimiterspace 
		#1 
		\vphantom{\big|} 
		\right|_{#2} 
}}

\newcommand{\Div}{\operatorname{div}}
\newcommand{\expnumber}[2]{{#1}\mathrm{e}{#2}}
\title{Shape Optimization for the Mitigation of Coastal Erosion via Shallow Water Equations}

\author{ \href{https://orcid.org/0000-0001-9930-065X}{\includegraphics[scale=0.06]{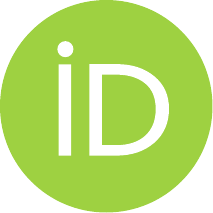}\hspace{1mm}Luka Schlegel} \\
	Department of Mathematics\\
	Universität Trier\\
	Universitätsring 15, 54296 Trier\\
	\texttt{schlegel@uni-trier.de} \\
	\And
	\href{https://orcid.org/0000-0001-7665-130X}{\includegraphics[scale=0.06]{orcid.pdf}\hspace{1mm}Volker Schulz} \\
	Department of Mathematics\\
	Universität Trier\\
	Universitätsring 15, 54296 Trier\\
	\texttt{volker.schulz@uni-trier.de} \\
}

\renewcommand{\headeright}{}
\renewcommand{\undertitle}{}
\renewcommand{\shorttitle}{SOMICE}

\hypersetup{
pdftitle={SWE Paper},
pdfsubject={q-bio.NC, q-bio.QM},
pdfauthor={Luka Schlegel, Volker Schulz},
pdfkeywords={Shape Optimization, Obstacle Problem,  Numerical Methods,  Adjoint Methods, Shallow Water Equations, Coastal Erosion},
}

\begin{document}
\maketitle

\begin{abstract}
Coastal erosion describes the displacement of land caused by destructive sea waves, currents or tides. Major efforts have been made to mitigate these effects using groins, breakwaters and various other structures. We address this problem by applying shape optimization techniques on the obstacles. We model the propagation of waves towards the coastline using two-dimensional shallow water equations with artificial viscosity. The obstacle's shape is optimized over an appropriate cost function to minimize the mechanical energy and to reduce velocities of water waves along the shore, without relying on a finite-dimensional design space, but based on shape calculus. 
\end{abstract}

\section{Introduction}
Coastal erosion describes the displacement of land caused by destructive sea waves, currents and/or tides. Major efforts have been made to mitigate these effects using groins, breakwaters and various other structures. 
Among experimental set-ups to model the propagation of waves towards a shore and to find optimal wave-breaking obstacles, the focus has turned towards numerical simulations due to the continuously increasing computational performance. Essential contributions to the field of numerical coastal protection have been made for steady \cite{Azerad2005}\cite{Mohammadi2008}\cite{Keuthen2015} and unsteady \cite{Mohammadi2011}\cite{Mohammadi2012} descriptions of propagating waves. In this paper we select one of the most widely applied system of wave equations. We describe the hydrodynamics by the set of Saint-Venant or better known as shallow water equations (SWE), that originate from the famous Navier-Stokes equations by depth-integration, based on the assumption that horizontal length-scales are much larger than vertical ones \cite{SaintVenant1871}. Calculating optimal shapes for various problems is a vital field, combining several areas of research. This paper builds up on the monographs \cite{Choi1987}\cite{Sokolowski1992}\cite{Delfour2011} to perform free-form shape optimization. In addition, we strongly orientate on \cite{Schulz2014b}\cite{Schulz2016}\cite{Schulz2016a} that use the Lagrangian approach for shape optimization, i.e. calculating state, adjoint and the deformation of the mesh via the volume form of the shape derivative assembled on the right-hand-side of the linear elasticity equation, as Riesz representative of the shape derivative. The calculation of the SWE continuous adjoint and shape derivative and its use in free-form shape optimization appears novel to us. However, we would like to emphasize, that the SWE have been used before in the optimization of practical applications, e.g. using discrete adjoints via automatic differentiation in the optimization of the location of tidal turbines \cite{Funke2014} and to optimize the shape of fish passages in finite design spaces \cite{AlvarezVazquez2006}\cite{Kadiri2019}.\\ The paper is structured as follows: In Section \ref{sec:PF} we formulate the PDE-constrained optimization problem. In Section \ref{sec:DerSha} we derive the necessary tools to solve this problem, by deriving adjoint equations and the shape derivative in volume form. The final part, Section \ref{sec:NumRes}, will then apply the results to firstly a simplified mesh and secondly to more realistic meshes, picturing first the Langue de Barbarie (LdB), a coastal section in the north of Dakar, Senegal that was severely affected by coastal erosion within the last decades and secondly a global illustration in the form of a spherical world mesh.

\section{Problem Formulation}\label{sec:PF}
Suppose we are given an open domain $\tilde{\Om}\subset\R^2$, which is split into the disjoint sets $\Om,D\subset\tilde{\Om}$ such that $\Om\cup D\cup\Ge=\tilde{\Om}$, $\Ga\cup\Gd=\partial\tilde{\Om}$.  We assume the variable, interior boundary $\Ge$ and the fixed outer $\partial\tilde{\Om}$ to be at least Lipschitz. One simple example of such kind is visualized below in Figure \ref{fig: domain}.

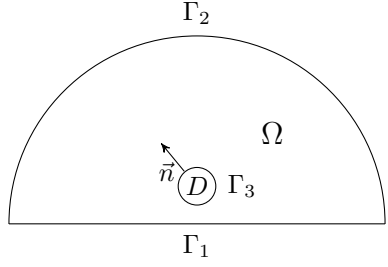
\begin{figure}[h]
	\centering
	\begin{tikzpicture}
	\begin{scope}
	\clip (0,0) rectangle (5,2.5);
	\draw (2.5,0) circle(2.5);
	\end{scope}

	\draw[->, >=stealth', shorten >=1pt] (2.33,0.7)   -- (2.0,1.1);
	\draw (0,0) --(5,0);
	\draw (2.5,0.5) circle (0.25);
	
	\node (A) at (3.5,1.2) {\large $\Omega$};
	\node (B) at (2.5,-0.3) {$\Ga$};
	\node (E) at (2.5,2.8) {$\Gd$};
	\node (F) at (3.1,0.5) {$\Ge$};
	\node (G) at (2.5,0.5) {$D$};
	\node (H) at (2.1,0.7) {$\nv$};
	\end{tikzpicture}
	\caption[Illustrative Domain]{Illustrative Domain $\Om$ with Initial Circled Obstacle $D$ and Boundaries $\Ga,\Gd$ and $\Ge$}
	\label{fig: domain}
\end{figure}

On this domain we model water wave and velocity fields as solution to SWE with artificial viscosity, i.e. 
\begin{equation}
\begin{aligned}
	\partial_tU+\nabla\cdot{F(U)}-\nab\cdot(G(\mu)\nab \hat{U})=S(U)\quad {}&{}\text{ in } \Om\times(0,T)\text{,}
	\label{Eq:1optprob}
\end{aligned}
\end{equation}
where we are given the SWE in vector notation with flux matrix 
\begin{equation}
\begin{aligned}
F(U)=\begin{pmatrix}
\Qv\\ 
\frac{\Qv}{H}\otimes \Qv+\frac{1}{2}gH^2\mathbf{I}_2
\end{pmatrix}
=
\begin{pmatrix}
Hu & vH\\	
Hu^2+\frac{1}{2}gH^2 & Huv \\
Huv & Hv^2 + \frac{1}{2}gH^2
\end{pmatrix}
\end{aligned}
\label{Eq:2Fluxes}
\end{equation}
for identity matrix $\mathbf{I}_2\in\R^{2\times2}$, gravitational acceleration $g$ and solution $U:\Om\times(0,T)\rightarrow\R\times\R^2$, where for simplicity the domain and time-dependent components are denoted by $U=(H,\Qv)=(H,Hu,Hv)$, with $H$ being the water height and $Hu,Hv$ the weighted horizontal and vertical discharge or velocity. For notational ease, we set $\hat{U}=(H+z,\Qv)$ for scalar sediment height $z:\Om\rightarrow\R$. The setting can be taken from Figure \ref{fig: SketchSWE}.
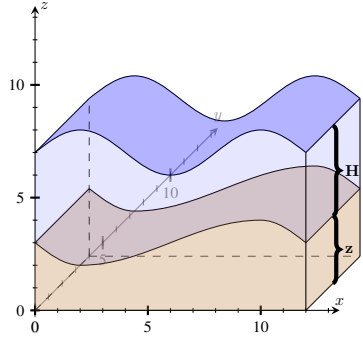
\begin{figure}[htb!]
	\centering
	\begin{tikzpicture}[x=0.5cm,y=0.5cm,z=0.3cm,>=stealth,scale=0.6, transform shape]
	\draw[->] (xyz cs:x=0) -- (xyz cs:x=13.5) node[above] {$x$};
	\draw[->] (xyz cs:y=0) -- (xyz cs:y=13.5) node[right] {$z$};
	\draw[->,opacity=0.4] (xyz cs:z=0) -- (xyz cs:z=13.5) node[above] {$y$};
	\foreach \coo in {0,1,...,13}
	{
		\draw (\coo,-1.5pt) -- (\coo,1.5pt);
		\draw (-1.5pt,\coo) -- (1.5pt,\coo);
		\draw[opacity=0.4] (xyz cs:y=-0.15pt,z=\coo) -- (xyz cs:y=0.15pt,z=\coo);
	}
	\foreach \coo in {0,5,10}
	{
		\draw[thick] (\coo,-3pt) -- (\coo,3pt) node[below=6pt] {\coo};
		\draw[thick] (-3pt,\coo) -- (3pt,\coo) node[left=6pt] {\coo};
		\draw[thick,opacity=0.4] (xyz cs:y=-0.3pt,z=\coo) -- (xyz cs:y=0.3pt,z=\coo) node[below=8pt] {\coo};
	}

	\draw[name path = A, black] (0,7) sin (2,8) cos(4,7) sin(6,6) cos(8,7) sin(10,8) cos(12,7);
	\draw[name path = B, black] (0,7,4) sin (2,8,4) cos(4,7,4) sin(6,6,4) cos(8,7,4) sin(10,8,4) cos(12,7,4);
	\draw[name path = C,black] (0,7) -- (0,7,4);
	\draw[name path = D,black] (12,7) -- (12,7,4);
	\tikzfillbetween[of=A and B]{blue, opacity=0.3};
	
	\draw[name path = E, black] (0,0) -- (12,0);
	\draw[name path = F, black,dashed,opacity=0.7] (0,0,4) -- (12,0,4);
	\draw[name path = G,black,dashed,opacity=0.7] (0,0) -- (0,0,4);
	\draw[name path = H,black] (12,0) -- (12,0,4);
	
	\draw[name path = I, black] (0,3) sin(2,2) cos(6,3) sin(10,4) cos(12,3);
	\draw[name path = J, black] (0,3,4) sin(2,2,4) cos(6,3,4) sin(10,4,4) cos(12,3,4);
	\draw[name path = K,black] (0,3) -- (0,3,4);
	\draw[name path = L,black] (12,3) -- (12,3,4);
	\tikzfillbetween[of=I and J]{brown, opacity=0.3};
	\tikzfillbetween[of=E and I]{brown, opacity=0.3};
	\tikzfillbetween[of=A and I]{blue, opacity=0.1};
	\tikzfillbetween[of=L and D]{blue, opacity=0.1};
	\tikzfillbetween[of=H and L]{brown, opacity=0.3};
	
	\draw[ black] (12,0) -- (12,7);
	\draw[ black] (12,0,4) -- (12,7,4);
	\draw[ black, dashed,opacity=0.7] (0,0,4) -- (0,7,4);

	\draw [decorate,ultra thick,
	decoration ={brace}] (12,3,2) -- (12,0,2) 
	node[pos=0.5,right=4pt,black,thick]{$\mathbf{z}$} ;
	\draw [decorate,ultra thick,
	decoration = {brace}] (12,7,2) -- (12,3,2) 
	node[pos=0.5,right=4pt,black,thick]{$\mathbf{H}$} ;
	
	\end{tikzpicture}
	\caption[Sketch of Wave and Sediment]{Cross-Section for Identification of Wave Height $H$ and Sediment Height $z$}
	\label{fig: SketchSWE}
\end{figure}
The source term in (\ref{Eq:1optprob}) is defined as  
\begin{equation}
\begin{aligned}\
S(U)=
\begin{pmatrix}
0\\
-gH\frac{\partial z}{\partial x}-gHu\frac{\sqrt{u^2+v^2}}{KH^{4/3}}\\
-gH\frac{\partial z}{\partial y}-gHv\frac{\sqrt{u^2+v^2}}{KH^{4/3}}
\end{pmatrix}\text{,}
\end{aligned}
\label{Eq:3Sources}
\end{equation}
where the first term responds to variations in the bed slope and the second term is resembling the Manning formula to respond to bottom friction, where $K>0$ is Manning's roughness coefficient \cite[Section 3.3.2]{Chow1959}.
For the boundaries we use rigid-wall and outflow conditions for $\Ga,\Ge$ and $\Gd$ by setting the velocity in normal direction to zero and prescribing a water height $H_1$ at the boundary, such that
\begin{equation}
\begin{aligned}
& \Qv\cdot \nv=0, \nab (H+z)\cdot\nv=0,\nab Q_1\cdot\nv=0,\nab Q_2\cdot\nv=0 {}&{}\text{ on } &&\Ga,\Ge&\times(0,T)&&\\
&H=H_1, \nab Q_1\cdot \nv=0, \nab Q_2\cdot \nv=0 {}&{}\text{ on }&& \Gd&\times(0,T).&&
\end{aligned}
\end{equation}
 Initial conditions for $U$ are implemented by prescribing a fixed starting point $U_0$, i.e.
\begin{equation}
\begin{aligned}
&U=U_0 {}&{}\text{ in }&& \Om&\times\{0\}&&
\end{aligned}
\end{equation}

\begin{remark}
Original viscous SWE are an incomplete parabolic system, where viscosity is only placed on the momentum equation. To prevent shocks or discontinuities that can appear in the original formulation of the hyperbolic SWE even for continuous data in finite time, an additional viscous term is added in the continuity equation such that we obtain a set of fully parabolic equations. We control the amount of added diffusion by the diagonal matrix $G(\mu)=\sum_{i=1}^ne_i^T\mu e_i e_i^T$ with entries $\mu=(\mu_v,\mu_f)\in\R_+\times\R^2_+$ and basis vector $e_i\in\R^n$ with $n$ being the number of dimensions in vector $\mu$. In this setting $\mu_f$ is fixed, while we rely on shock detection in the determination of $\mu_v$ following \cite{Persson2006}.
	Ultimately, a physical interpretation can be obtained for the introduction of the viscous part in the conservation of momentum equations. However, $\mu_v$ is solely based on stabilization arguments, where we follow the justification as in \cite{Guba2014}. The complete parabolic problem together with well-posed boundary conditions \cite{Oliger1978} provides us with a well-posed problem.
\end{remark} 

We obtain a PDE-constrained optimization problem for objective
\begin{align}
J(\Om)=J_1(\Om)+J_2(\Om)+J_3(\Om)+J_4(\Om)+J_5(\Om)\text{,}\label{Eq:SWEObj}
\end{align}
where we are trying to minimize the mechanical wave energy of destructive waves at the shore $\Ga$, that are waves  above a critical threshold $H_{\text{cr}}>0$ \cite{Mohammadi2008}, over a time window $\tilde{T}\subset(0,T)$, i.e.
\begin{align}
J_1(\Om)=\int_{\tilde{T}}\int_{\Gamma_1}&\nu_1E\sigma_{\alpha}(H-H_{\text{cr}})\diff s \diff t
\label{Eq:1ObjSWE}
\end{align}
for mechancial wave energy $E=\frac{1}{8}\rho g H^2$ and reduction to destructive sea waves enforced by usage of the sigmoid function $\sigma_\alpha:\R\rightarrow\R$ with slope parameter $\alpha>0$.
In addition, we aim for zeroed  velocities
\begin{align}
J_2(\Om)=\int_0^T\int_{\Gamma_1}\frac{\nu_2}{2}||\Qv||_2^2\diff s\diff t.
\label{Eq:1Obj2}
\end{align}
These objectives are supplemented by a volume penalty and a perimeter regularization, i.e. 
\begin{align}
J_3(\Om)=-\nu_3\int_{\Om}1\diff x\text{,}
\label{Eq:2volSWE}
\end{align}
and
\begin{align}
J_4(\Om)=\nu_4\int_{\Gamma_{3}}1\diff s.
\label{Eq:1PeriSWE}
\end{align}
Additionally, a minimal thinness penalty on obstacle level is added by following \cite{Allaire2016} as
\begin{align}
J_5(\Om)=\nu_5\int_{\Gamma_3}\int_0^{d_{min}}\left[(d_{\Om}\left(x-\xi\vec{n}(x)\right))^+\right]^2\diff\xi \diff s\text{.}
\label{Eq:1Thick}
\end{align}
Here $d_\Om$ represents the signed distance function (SDF) with value 
\begin{align}
d_{\Om}(x)=\begin{cases}
d(x,\partial\Om)\quad &\text{ if } x\in \Om\\
0 \quad &\text{ if } x\in \partial\Om
\\
-d(x,\partial\Om)\quad &\text{ if } x\in \bar{\Om}^c\text{,}
\end{cases}
\label{Eq:SDF}
\end{align}
where the Euclidian distance of $x\in\R^d$ to a closed set $K\subset\R^d$ is defined as
\begin{align}
d(x,K)=\min_{y\in K}||x-y||_2
\label{Eq:6EuclMin}
\end{align}
for Euclidian distance $||.||_2$. The latter penalty can be justified by arguing, that an increased thinness would be undesirable with regards to the durability of the optimized shape. From a shape computational viewpoint, it ensures staying in the associated shape space. In numerics it prevents intersections of line segments, which may cause a breakdown of the optimization algorithm. In this light, we only take into account the positive part of the SDF of the offset value. Hence, we define for a real-valued function $f:\Om\rightarrow\R$ the positive part as
\begin{align}
f^+=\max(f(x),0)=\begin{cases}
f(x)\quad&\text{ if }f(x)>0\\
0\quad&\text{ otherwise }
\end{cases}\text{.}
\end{align}
 Finally, we would like to point out, that the objective is controlled by parameters $\nu_1,\nu_2,\nu_3,\nu_4$ and $\nu_5$ which need to be defined a priori (for further details cf. to Section  \ref{sec:NumRes}).
\begin{remark}
	The volume penalization could also be replaced by a geometrical constraint to meet a certain voluminous value, e.g. the initial size of the obstacle 
	\begin{align*}
	\int_\Om 1\diff x=\text{vol}(\Om)=\text{vol}(\Om_0)=\int_{\Om_0} 1\diff x\text{.}
	\end{align*}
	This approach would call for a different algorithmic handle, e.g. in \cite{Schulz2016a} an augmented Lagrangian is proposed.
\end{remark}

\section{Derivation of the Shape Derivative}\label{sec:DerSha}
We now fix notations and definitions in the first part, before deriving the adjoint equations and shape derivatives in the second part, that are necessary to solve the PDE-constrained optimization problem.
\subsection{Notations and Definitions}
The idea of shape optimization is to deform an object ideally to minimize some target functional. Hence, to find a suitable way of deforming we are interested in some shape analogy to classical derivatives. Here we use a methodology that is commonly used in shape optimization, extensively elaborated in various works \cite{Choi1987}\cite{Sokolowski1992}\cite{Delfour2011}.\\
In this section we fix notations and definitions following \cite{Schulz2016}\cite{Schulz2016a}, amending whenever it appears necessary.
We start by introducing a family of mappings $\{\phi_\eps\}_{\eps\in[0,\tau]}$ for $\tau>0$ that are used to map each current position $\x\in\Om$ to another by $\phi_\eps(\x)$, where we choose the vector field $\Vv$ as the direction for the so-called perturbation of identity
\begin{align}
\x_\eps=\phi_\eps(\x)=\x+\eps \Vv(\x)\text{.}
\label{Eq:4poi}
\end{align}
According to this methodology, we can map the whole domain $\Om$ to another $\Omeps$ such that 
\begin{align}
\Omeps=\{\x_\eps|x+\eps \Vv(x),x\in\Om\}\text{.}
\label{Eq:5domain}
\end{align}
We define the Eulerian Derivative as 
\begin{align}
DJ(\Om)[\Vv]=\lim_{\eps\rightarrow 0^+} \frac{J(\Om_\eps)-J(\Om)}{\eps}\label{Eq:7EulerDer}\text{.}
\end{align}
Commonly, this expression is called shape derivative of $J$ at $\Om$ in direction $\Vv$ and in this sense $J$ shape differentiable at $\Om$ if for all directions $\Vv$ the Eulerian derivative exists and the mapping $\Vv\mapsto DJ(\Om)[\Vv]$ is linear and continuous.
In addition, we define the material derivative of some scalar function $p:\Om\rightarrow\R$ at $x\in\Om$ by the derivative of a composed function $p_\eps\circ\phi_\eps:\Om\rightarrow\Om_\eps\rightarrow\R$ for $p_\eps:\Om_\eps\rightarrow\R$ as
\begin{align}
D_m p(x):=\lim_{\eps\rightarrow0^+}\frac{p_\eps\circ \phi_\eps(x)-p(x)}{\eps}=\frac{d}{d\eps}\restr{(p_\eps\circ \phi_\eps)(x)}{\eps=0^+}\label{Eq:8MatDer}
\end{align} 
and the corresponding shape derivative for a scalar $p$ and a vector-valued $\Pv$ for which the material derivative is applied component-wise as
\begin{align}
	Dp[\Vv]:=&D_mp-\Vv\cdot\nab p \label{Eq:9MatDer2}\\
	D\Pv[\Vv]:=&D_m\Pv-\Vv^T\nab \Pv \label{Eq:9MatDer2vec}\text{,}
\end{align}
where the distinction is that $\nab p$ is the gradient of a scalar and $\nab\Pv$ is the tensor derivative of a vector.
In the following, we will use the abbreviation $\dot{p}$ and $\dot{P}$ to mark the material derivative of $p$ and $P$. In Section \ref{sec:DerSha} we will need to have the following calculation rules on board \cite{Berggren2010}
\begin{align}
	D_m(pq)&=D_mpq+pD_mq\label{Eq:10MatProdR}\\
	D_m\nab p&=\nab D_mp-\nab \Vv^T\nab p\label{Eq:11MatGradR}\\
	D_m\nab \Pv&=\nab D_m\Pv-\nab \Vv^T\nab \Pv\label{Eq:11MatGradRvec}\\
	D_m(\nab q^T\nab p)&=\nab D_mp^T\nab q-\nab q^T(\nab \Vv+\nab \Vv^T)\nab p+\nab p^T\nab D_mq \label{Eq:12MatGradProdR}\text{.}
\end{align}
In addition, the basic idea in the proof of the shape derivative in the next section will be to pull back each integral defined on the transformed field back to the original configuration. We therefore need to state the following rule for differentiating domain integrals \cite{Berggren2010}
\begin{align}
	\frac{d}{d\eps}\restr{\left(\int_{\Om_{\eps}}p_\eps\diff x_\eps\right)}{\eps=0^+}=\int_\Om(D_mp+\nab\cdot \Vv p)\diff x \label{Eq:13DoaminR}\text{.}
\end{align}

\subsection{Shape Derivative}
From the discussion above, we define the derivative of some functional with respect to $\Om$ in the $\Vv$ direction that explicitly and implicitly depends on the domain $j(\Om,u(\Om))$ by
\begin{align}
Dj(\Om,u(\Om))[\Vv]=\frac{d}{d_\eps}j(\Omeps,u(\Omeps))|_{\eps=0}=D_1j(\Om,\usol(\Om))[\Vv]+D_2j(\Om,\usol(\Om))\dot\usol \label{Eq:14JDer}\text{,}
\end{align}
where 
\begin{align}
\dot\usol=\frac{d}{d_\eps}\usol(\Omeps)|_{\eps=0}\label{Eq:15uDer}\text{.}
\end{align}
The idea is to circumvent the derivative of $\usol$, which would imply one problem for each direction of $\Vv$ by solving an auxiliary problem \cite{Schulz2014b}.\\
Before defining this problem, we take care of the constraints (\ref{Eq:1optprob}) and formulate the Lagrangian
\begin{align}
\mathcal{L}(\Om,U,P) = J_{1,2}(\Om)+a(U,P)-b(P) \label{Eq:16LagSWE}\text{,}
\end{align}
where $J_{1,2}(\Om)=J_{1}(\Om)+J_{2}(\Om)$ consists of the first two objectives (\ref{Eq:1ObjSWE})-(\ref{Eq:1Obj2}), and $a(U,P)$ and $b(P)$ are obtained from the boundary value problem (\ref{Eq:1optprob}).
We rewrite the equations in weak form by multiplying with some arbitrary test function $P\in H^1(\Om\times(0,T))^3$ obtaining the form $a(U,P)=a(H,\Qv,p,\rv)$
\begin{equation}
\begin{aligned}
a(H,\Qv,p,\rv):=&\int_0^T\int_\Om\left[\frac{\partial H}{\partial t}+\nab\cdot\Qv\right]p\diff x\diff t\\
+&\int_0^T\int_\Om\left[\frac{\partial \Qv}{\partial t}+\nab\cdot\left(\frac{\Qv}{H}\otimes \Qv+\frac{1}{2}gH^2\mathbf{I}_2\right)\right]\cdot\rv\diff x\diff t\\
+&\int_0^T\int_\Om\mu_v\nab (H+z)\cdot\nab p\diff x\diff t-\int_0^T\int_{\Gd}\mu_v\nab (H_1+z)\cdot\nv p\diff s\diff t\\
+&\int_0^T\int_\Om G(\mu_f)\nab\Qv:\nab\rv\diff x\diff t+\int_0^T\int_\Om gH\nab z\cdot\rv\diff x\diff t\\
\label{Eq:17aweakSWE}
\end{aligned}
\end{equation}
and a zero perturbation term.
\begin{remark}
	For readability we left out the friction term, however up to some repetitive use of chain and product rule the handling stays the same as for the variations in the bed slope.
\end{remark}
\begin{remark}
	Here and in what follows we assume the flow to be free of discontinuities, e.g. induced by a discontinuous bottom profile $z$ or wave height $H$, which would prohibit us from performing adjoint-sensitivity analyses and ensuring the requirements in Theorem \ref{AdjointTheo} and \ref{ShapeDerTheo}.
\end{remark}
\begin{remark}
	To continue with adjoint calculations and to enforce initial and boundary conditions we are required to integrate by parts on the derivative-containing terms. 
\end{remark}
We obtain state equations from differentiating the Lagrangian with respect to $P$ and the auxiliary problem, the adjoint equations, from differentiating the Lagrangian with respect to the states $U$.
The adjoint is formulated in the following theorem: 
\begin{theorem} \label{AdjointTheo} 
	(Adjoint) Assume that the parabolic PDE problem (\ref{Eq:1optprob}) is $H^1$-regular, so that its solution $U$ is at least in $H^1(\Om\times(0,T))^3$. Then the adjoint in strong form (without friction term) is given by
	\begin{equation}
	\begin{aligned}
	-\frac{\partial p}{\partial t}+\frac{1}{H^2}(\Qv\cdot\nab)\rv\cdot\Qv-gH(\nab\cdot\rv) -\nab\cdot(\mu_v\nab p)+g\nab z\cdot \rv&=-\nu_1(E\sigma_\alpha)_{H,\Ga,\tilde{T}} \\
	-\frac{\partial\rv}{\partial t}-\nab p-\frac{1}{H}(\Qv\cdot\nab)\rv-\frac{1}{H}(\nab\rv)^T\Qv-\nab\cdot( G(\mu_f)\nab \rv)&=-\nu_2(\Qv)_{\Ga}
	\end{aligned}
	\label{Eq:19Adjoint}
	\end{equation}
	where we have on $\Gamma_1\times\tilde{T}$
	\begin{equation}
	\begin{aligned}
	(E\sigma_\alpha)_{H,\Ga,\tilde{T}} =2\frac{E}{H}\sigma_{\alpha}(H-H_{\text{cr}})+E\sigma_{\alpha}(H-H_{\text{cr}})(1-\sigma_{\alpha}(H-H_{\text{cr}}))
	\end{aligned}
	\end{equation}
		\text{such as final time conditions}
	\begin{equation}
	\begin{aligned}
	p&=0& \text{ in}&&\Om&\times\{T\}&&\\
	\rv&=0  & \text{ in}&&\Om&\times\{T\}&&
		\end{aligned} 
	\end{equation}
	and boundary conditions
		\begin{equation}
	\begin{aligned}
	\rv\cdot \nv=0, \nab p\cdot \nv=0,\nab \rv_1\cdot\nv=0, \nab \rv_2\cdot\nv&=0& \text{ on}&&\Ga,\Ge&\times(0,T)&&\\
	p\nv+\frac{1}{H_1}(\Qv\cdot\nv)\rv+\frac{1}{H_1}(\Qv\rv) \cdot \nv=0,\nab \rv_1\cdot\nv=0, \nab \rv_2\cdot\nv&=0 & \text{ on}&&\Gd&\times(0,T)\text{.}&&
	\end{aligned} 
	\label{Eq:20AdjointBC}
	\end{equation}
\end{theorem}

\begin{proof}
	See Appendix \ref{app:derivadj}
\end{proof}
The obtained adjoint equations can be written in vector form as 
\begin{align}
-\frac{\partial P}{\partial t}+AP_x+BP_y+CP-\nab\cdot(G(\mu)\nab P)=S\text{,} \label{Eq:AdjVectForm}
\end{align}
where
\begin{align}
A=\begin{pmatrix}
0&\frac{Q_1}{H^2}-gH&\frac{Q_1Q_2}{H^2}\\
-1&-2\frac{Q_1}{H}&-\frac{Q_2}{H}\\
0&0&-\frac{Q_1}{H}
\end{pmatrix},\quad	
B=\begin{pmatrix}
0&\frac{Q_1Q_2}{H^2}&\frac{Q_2^2}{H^2}-gH\\
0&-\frac{Q_2}{H}&0\\
-1&-\frac{Q_1}{H}&-2\frac{Q_2}{H}
\end{pmatrix}
\end{align}
and $C$ originates from variations in the sediment in (\ref{Eq:3Sources}) such that
\begin{align}
C=\begin{pmatrix}
0&g\frac{\partial z}{\partial x}&g\frac{\partial z}{\partial y}\\
0&0&0\\
0&0&0
\end{pmatrix}\text{.}
\label{AdjVectorNotSource}
\end{align}
Finally, $S$ corresponds to the right hand-side of (\ref{Eq:19Adjoint}).
\begin{remark}
	If one desires to include additional sources, e.g. accounting for sediment friction, $C$ from (\ref{AdjVectorNotSource}) would need to be adjusted. 
\end{remark}
\begin{remark}
	Shape derivatives can for a sufficiently smooth domain be described via boundary formulations using Hadamard's structure theorem \cite{Sokolowski1992}. The integral over $\Om$ is then replaced by an integral over $\Ge$ that acts on the associated normal vector. In this paper, we will only consider the volume form, which will be then used to obtain smooth mesh deformations from a Riesz projection of this shape derivative.
\end{remark}

\begin{theorem}\label{ShapeDerTheo}
		(Shape Derivative)
     Assume that the parabolic PDE problem (\ref{Eq:1optprob}) is $H^1$-regular, so that its solution $U$ is at least in $H^1(\Om\times(0,T))^3$. Moreover, assume that
		the adjoint equation (\ref{Eq:19Adjoint}) admits a solution $P\in H^1(\Om\times(0,T))^3$. Then the
			shape derivative of the objectives $J_{1,2}$ (without friction term) at $\Om$ in the
			direction $\Vv$ is given by
	\begin{equation}
	\begin{aligned}
	DJ_{1,2}(\Om)[\Vv]=\int_{0}^{T}\int_{\Om}&\Big[-(\nab \Vv)^T:\nab \Qv p - (\nab \Vv)^T:\nab \Qv\frac{\Qv}{H}\cdot \rv- (\nab \Vv\Qv\cdot\nab)\frac{\Qv}{H}\cdot \rv \\ 
	-& gH(\nab \Vv)^T\nab H \cdot \rv-\mu_v\nab (H+z)^T(\nab \Vv +\nab \Vv^T)\nab p  \\
	-&G(\mu_f)\nab \Qv \nab \Vv:\nab \rv -G(\mu_f)\nab\Qv\nab \Vv^T:\nab\rv  \\
	-&gH\nab \Vv^T\nab z\cdot\rv+\Div(\Vv)\Big\{\frac{\partial H}{\partial t}p+\nab\cdot \Qv p+\frac{\partial \Qv}{\partial t}\cdot \rv\\
+	& (\Qv\cdot \nab)\frac{\Qv}{H}\cdot \rv +\nab\cdot \Qv\frac{\Qv}{H}\cdot \rv +  \frac{1}{2}g\nab H^2\cdot \rv +gH\nab z\cdot\rv\\
	+&\mu_v\nab (H+z)\cdot\nab p +G(\mu_f)\nab \Qv : \nab \rv \Big\}\Big]\diff x\diff t\text{.}
	\end{aligned}
	\label{Eq:21SDSWE}
	\end{equation}
\end{theorem}

\begin{proof}
	See Appendix \ref{app:derivsha}
\end{proof}

The shape derivatives of the penalty terms (volume, perimeter and thickness) are obtained as, see e.g. \cite{Sokolowski1992}\cite{Welker2017}
\begin{align}
DJ_{3}(\Om)[\Vv]&=\nu_3\int_{\Om}-\nab\cdot\Vv\diff x\\
DJ_{4}(\Om)[\Vv]&=\nu_4\int_{\Ge}\kappa_m\langle \Vv,\nv\rangle \diff s\label{Eq:30DJreg}
\end{align}
and see \cite{Allaire2016} for
\begin{equation}
\begin{aligned}
&DJ_5(\Om)[\Vv]&=\nu_5\int_{\Ge}\int_0^{d_{min}}\Big[&\Vv(x)\cdot\vec{n}(x)\Big\{\kappa_m(x)(d_{\Om}\left(x_m\right)^+)^2&\\
&&+&2d_{\Om}(x_m)^+\nab d_\Om(x_m)\cdot\nab d_\Om(x)\Big\}&\\
&&-&\Vv(p_{\partial\Om}(x_m))\cdot\vec{n}(p_{\partial\Om}(x_m))2(d_{\Om}(x_m))^+\Big]\diff\xi \diff s&
\label{Eq:ShdThick}
\end{aligned}
\end{equation}
for mean curvature $\kappa_m$, and offset point $x_m=x-\xi\nv(x)$, where we require the shape derivative of the SDF \cite{Allaire2016}
\begin{align}
Dd_\Om(x)[\Vv]=-\Vv(p_{\partial\Om}(x))\cdot\vec{n}(p_{\partial\Om}(x))\label{EqSDtoSDF}
\end{align}
with operator $p_{\partial\Om}$ that projects a point $x\in\Om$ onto its closest boundary and holds for all $x\notin\Sigma$, where $\Sigma$ is referred to as the ridge, where the minimum in (\ref{Eq:6EuclMin}) is obtained by two distinct points.

\section{Numerical Results} \label{sec:NumRes}
We now first discuss the implementation in detail, before applying these techniques to selected examples in the following subsections.
\subsection{Implementation Details}
We rely on the classical structure of adjoint-based shape optimization algorithms shortly sketched in the algorithm below.
\begin{algorithm}[htb!]
	\caption{Shape Optimization Algorithm}
	\begin{algorithmic}
		\STATE Initialization
		\WHILE  {$||DJ(\Om_k)[\Vv]||>\eps_{TOL}$}
		\STATE 1. Calculate SDF $w_k$ [via AABBT]
		\STATE 2. Calculate State $U_k$ [via (\ref{Eq:22DGDiscretization})]
		\STATE 3. Calculate Adjoint $P_k$ [via (\ref{Eq:22DGDiscretization})]
		\STATE 4. Calculate Gradient $W_k$ [via $DJ_{1,2,3,4,5}(\Om)[\Vv]$ \& Linear Elasticity (\ref{Eq:29LinearElasticity})]
		\STATE 5. Perform Linesearch for $\tilde{W}_k$
		\STATE 6. Calculate $\Om_{k+1}$ [via $\tilde{W}_k$ and (\ref{Eq:5domain})]
		\ENDWHILE
	\end{algorithmic}
\end{algorithm}
The solution to the SDF in (\ref{Eq:1Thick}) is mesh dependent. For a mesh with undiscretized obstacle the SDF is approximated based on axes-aligned-bounding-boxes trees (AABBT) \cite{FEniCS2015} on a background mesh. We refer to Figure \ref{fig:SDF-Comparison} for an exemplifying visualization. Note, we have highlighted the initial boundary mesh points in red, exemplifying offset points in blue such as mesh and background mesh in the left figure and due to visibility, the distance of background nodes to the nearest exterior boundary point of the original mesh in the right figure.
\begin{figure}[h]
\begin{center}
	\includegraphics[scale=0.3]{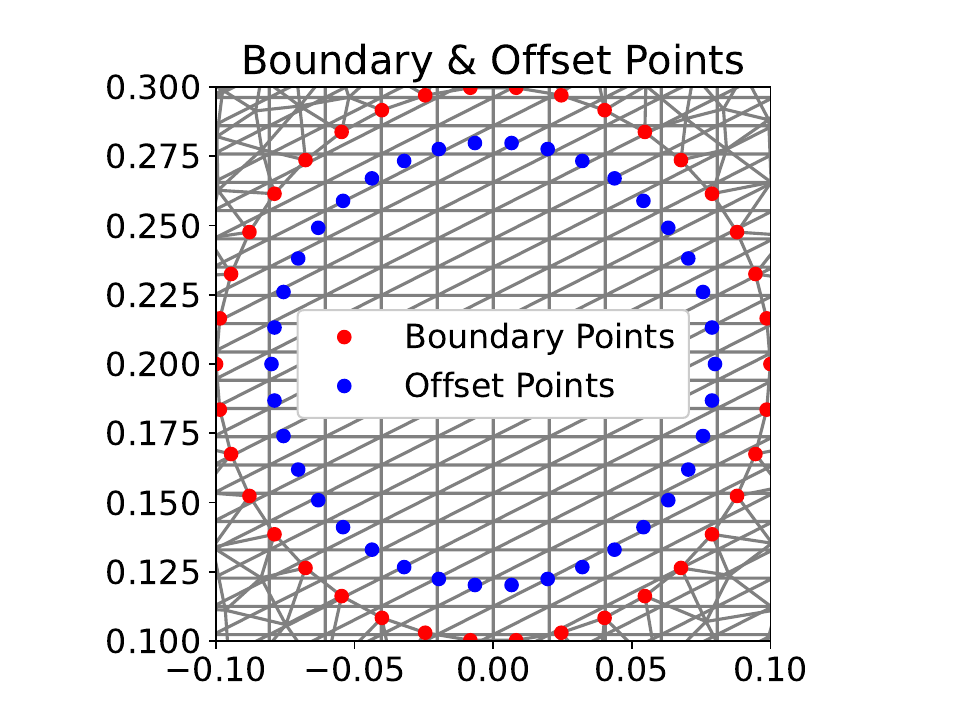}
	\includegraphics[scale=0.3]{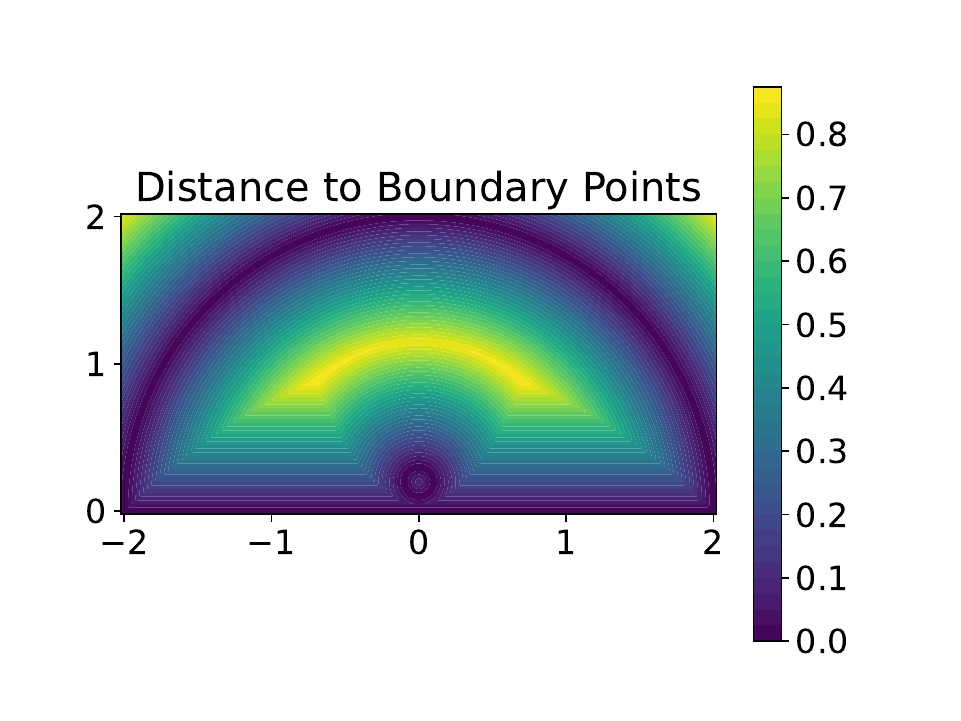}
\end{center}
\vspace{-0.5cm}
\caption[SDF]{1.: Boundary and Offset Points on Mesh and Background Mesh, 2.: Distance to Boundary Points via AABBT}
\label{fig:SDF-Comparison}
\end{figure}

We solve the boundary value problem (\ref{Eq:1optprob}), the adjoint problem (\ref{Eq:19Adjoint}) and the deformation of the domain with the help of the finite element solver FEniCS \cite{FEniCS2015}. For the time discretization we can choose between implicit and explicit integration arising from theta-methods \cite{Beam1982}.
High accuracy even for the inviscid and hyperbolic PDE, i.e. $\mu=0$, is achieved using a discontinuous Galerkin (DG) method to discretize in space \cite{Aizinger2002}\cite{Karna2011}\cite{Khan2014}. This implies discontinuous cell transitions, and hence a formulation based on each element $\kappa\in\mathcal{T}_h$ or facet $\Gamma_I$ for a subdivision $\mathcal{T}_h$ of some domain $\Om$, such as a redefinition of each function and operator on the so-called broken and possibly vector-valued $d$-dimensional Sobolov space $\mathcal{H}^1(\mathcal{T}_h\times(0,T))^d$. In this light, we also need to define the average $\{\!\{U\}\!\}=(U^++U^-)/2$ and jump term $\underline{[\![U]\!]}=U^+\otimes n_++U^-\otimes n_-$ to express fluxes on cell transitions. The discretization then reads for solution and test-function $U_h,P_h$ from some finite element approximation space of $\mathcal{H}^1(\mathcal{T}_h\times(0,T))^3$ for an SIPG scheme as \cite{Hartmann2008}\cite{Houston2018}\begin{equation}
\begin{aligned}
F_h(U_h,P_h)=&\int_{0}^{T}\int_{\Om}\Big[\frac{\partial U_h}{\partial t}\cdot P_h-F(U_{h}):\nab_hP_h+G(\mu)\nab_h(\hat{U}_h):\nab_hP_h\\
-&S(U_h)\cdot P_h\Big] \diff x\diff t+\int_{0}^{T}\sum_{\kappa\in\mathcal{T}_h}\int_{\partial\kappa\setminus\Gamma}\mathcal{F}(U^{+}_h,U^{-}_h,\nv)\cdot P^{+}_h\diff s\diff t\\
+&\int_{0}^{T}\int_{\Gamma_I}\Big[\underline{\delta}_h:\underline{[\![P_h]\!]}-\{\!\{G(\mu)\nab_h (P_h)\}\!\}:\underline{[\![\hat{U}_h]\!]}\\
-&\{\!\{G(\mu)\nab_h( \hat{U}_h)\}\!\}:\underline{[\![P_h]\!]}\Big] \diff s\diff t+N_{\Gamma,h}(U_h,P_h)=0\text{,}
\label{Eq:22DGDiscretization}
\end{aligned}
\end{equation}
where the numerical flux function $\mathcal{F}(U^{+}_h,U^{-}_h,\nv)$ defines the fluxes at the discontinuous cell transitions, incorporating specific quantities at the respective boundaries. For the advective flux and for a given flux Jacobian $\mathcal{J}_i:=\partial_UF_i(U)$ and matrix $B(U,\nv)=\sum_{i=1}^{2}n_i\mathcal{J}_i(U)$ we can choose between a variety of numerical fluxes \cite{Aizinger2002}, e.g.\\

(Local) Lax-Friedrichs Flux:
\begin{equation}
\begin{aligned}
\mathcal{F}_{1}(U^+,U^-,\nv)|\partial\kappa=\frac{1}{2}\left(F(U^+)\cdot \nv+F(U^-)\cdot \nv+\alpha_{\max}(U^+-U^-)\right)\text{,}	\end{aligned}
\label{Eq:23LaxFrFlux}
\end{equation}
where $\alpha_{\max}=\max_{V=U^+,U^-}\{|\lambda(B(V,\nv_\kappa))|\}$ with $\lambda(B(V,\nv_\kappa))$ returning a sequence of eigenvalues for the matrix $B$ restricted on a side of element $\kappa$.\\

HLLE Flux:
\begin{equation}
\begin{aligned}
\mathcal{F}_{2}(U^+,U^-,\nv)|\partial\kappa=\frac{1}{\lambda^+-\lambda^-}\left(\lambda^+F(U^+)\cdot\nv-\lambda^-F(U^-)\cdot \nv-\lambda^+\lambda^-(U^+-U^-)\right)	\text{,}\end{aligned}
\label{Eq:24HLLEFlux}
\end{equation}
where $\lambda^+=\max(\alpha_{\max},0)$ and $\lambda^-=\min(\alpha_{\min},0)$, for $\alpha_{\min}$ defined in accordance with $\alpha_{\max}$.
The required SWE Jacobian is written as
\begin{equation}
\begin{aligned}
\hspace{-0.8cm}
\mathcal{J}_1(U)=
\begin{pmatrix}
0&1&0\\
-\frac{Q_1^2}{H^2}+gH&2\frac{Q_1}{H}&0\\
-\frac{Q_1Q_2}{H^2}&\frac{Q_2}{H}&\frac{Q_1}{H}\\
\end{pmatrix}
\quad \mathcal{J}_2(U)=
\begin{pmatrix}
0&0&1\\
-\frac{Q_1Q_2}{H^2}&\frac{Q_2}{H}&\frac{Q_1}{H}\\
-\frac{Q_2^2}{H^2}+gH&0&2\frac{Q_2}{H}
\end{pmatrix} \text{.}
\end{aligned}
\label{Eq:25Jac}
\end{equation}
Hence, we obtain the following eigenvalues, where $c=\sqrt{gH}$ denotes the wave celerity \cite{Aizinger2002}

\begin{equation}
\begin{aligned}
\lambda(n_1\mathcal{J}_1+n_2\mathcal{J}_2)&=\{\lambda_{1},\lambda_2,\lambda_3\}
\\&=\{un_1+vn_2-c,un_1+vn_2,un_1+vn_2+c\}\text{.}
\end{aligned}
\label{Eq:26EV}
\end{equation}
\begin{remark}
	From (\ref{Eq:26EV}) also the hyperbolicy for the shallow water system is obtained, i.e. $\lambda_{i}\in\R$ for $i\in\{1,...,3\}$. In addition if $c\neq0$ or $H>0$, we obtain distinct eigenvalues, which lead to strict hyperbolicy.
\end{remark}
\begin{remark}
For a mesh with discretized obstacle and suitable transitional boundaries the SDF can be based on the solution of the diffusive Eikonal Equation with $f(x)=1$, $q(x)=0$
	\begin{equation}
	\begin{aligned}
	|\nab w(x)|-\mu_{SDF}\Delta w(x)&=f(x) \quad &x&\in\Om\\
	w(x)&=q(x)\quad &x&\in\partial\Om\text{,}
	\label{Eq:Eikonal}
	\end{aligned}
	\end{equation}
	written in weak form as
	\begin{equation}
	\begin{aligned}\int_\Om\sqrt{\nab w\cdot\nab w}v\diff x-\int_\Om fv\diff x +\int_\Om\mu_{SDF} \nab w\cdot\nab v\diff x=0\text{,}
	\label{Eq:EikonalViscous}
	\end{aligned}
	\end{equation}
where $w\in H^1(\Om)$ for all $v\in H^1(\Om)$ and $\mu_{SDF}=\max_i h_i$ is dependent on the cell-diameter $h_i$ for the $i^{th}$ cell $\kappa_i\subset\Om$ for $i\in\{1,...,m\}$. In this setting, the diffusive Eikonal equation can serve as an additional constraint to  (\ref{Eq:SWEObj}) and be considered in adjoint-based shape optimization.
\end{remark}
\begin{remark}
	In the presence of sources, especially for a discontinuous sediment $z$, a well-balanced numerical scheme is only obtained by methods of flux balancing. For this, the method presented in \cite{Wang2006} is extended to two dimensions. In addition, diffusive terms introduced in (\ref{Eq:1optprob}) cancel naturally in still water conditions. Finally respectively (\ref{Eq:23LaxFrFlux}) and (\ref{Eq:24HLLEFlux}) are redefined.
\end{remark}
In (\ref{Eq:22DGDiscretization}) we define the penalization term for the viscous fluxes as
\begin{equation}
\begin{aligned}
\underline{\delta}_h(\hat{U}_h)=C_{IP}\frac{k^2}{h}\{\!\{G(\mu)\}\!\}\underline{[\![\hat{U}_h]\!]}\text{,}
\end{aligned}
\label{Eq:27viscousflux}
\end{equation}
where $C_{IP}>0$ is a constant, $k>0$ the polynomial order of the DG method and $h>0$ the ratio of the cell volume and the facet area.
What is remaining in (\ref{Eq:22DGDiscretization}) is the specification of the boundary term, here we state that
\begin{equation}
\begin{aligned}
N_{\Gamma,h}(U_h,P_h)=&\int_{0}^T\int_{\Gamma}\mathcal{F}(U^{+}_h,U_\Gamma(U^{+}_h),\nv)\cdot P^{+}_h\diff s\diff t\\
+&\int_{0}^T\int_{\Gamma_N}\Big[\underline{\delta}_\Gamma (\hat{U}_h^+):P_h\otimes \nv+ G(\mu^+)\nab_h(\hat{U}^+_h):P_h^+\otimes\nv\\
-&G(\mu^+)\nab_hV_h^+:(\hat{U}_h^+-U_{\Gamma}(\hat{U}_h^+))\otimes \nv\Big] \diff s\diff t\text{,}
\end{aligned}
\label{Eq:28DGDiscretizationBoundary}
\end{equation}
where $\Gamma_N$ are all boundaries of type Neumann. Additionally, we define
\begin{align}
\underline{\delta}_\Gamma(U_h^+)&=C_{IP}G(\mu^+)\frac{k^2}{h}(U_h^+-U_{\Gamma}(U_h^+))\otimes\nv\\
\mathcal{F}(U^{+}_h,U_\Gamma(U^{+}_h),\nv)&=\frac{1}{2}[\nv\cdot F(U_h^+)+\nv\cdot F(U_{\Gamma}(U_h^+))]\text{.}
\end{align} For the pure advective SWE open and rigid-wall boundary functions are defined as in \cite{Aizinger2002}.
Having obtained a discretized solution for the forward problem, we calculate the SWE adjoint problem in the same manner using a DG discretization in space and a member of the theta-method for the time discretization. For this we rewrite the vector form of the SWE adjoint (\ref{Eq:AdjVectForm}) with the help of the product rule, i.e.
\begin{align}
\frac{\partial P}{\partial t}-\nab\cdot(AP,BP)-\tilde{C}P+\nab\cdot(G(\eps)\nab P)=-S \text{,}
\end{align}
where $\tilde{C}$ is defined to be
\begin{align}
\tilde{C}=C-A_x-B_y\text{.}
\end{align} 
The following theorem provides us then with the necessary eigenvalues of the adjoint flux Jacobian $\mathcal{J}^*_i:=\partial_PF^*_i(P)=-\partial_P(AP,BP)$.
\begin{theorem}(Eigenvalues of the Adjoint Flux Jacobian)
	The eigenvalues of matrix $B^*(P,\nv)$ belonging to the adjoint flux Jacobian $\mathcal{J}^*_i:=\partial_PF^*_i(P)$ equal the eigenvalues of matrix $B(U,\nv)$ belonging to the flux Jacobian $\mathcal{J}_i:=\partial_UF_i(U)$. \label{Theo:AdjointFlux}
\end{theorem}
\begin{proof}
	\begin{align}
	\lambda(B(U,\nv))=\lambda(\sum_{i=1}^{2}\nv_i\partial_UF_i(U))=\lambda(\sum_{i=1}^{2}n_i\partial_PF^*_i(P))=\lambda(B^*(P,\nv))
	\end{align}
	since $\sum_{i=1}^{2}n_i\partial_UF_i(U)=\sum_{i=1}^{2}n_i\partial_PF^*_i(P)^T$ which is due to the linearity of the adjoint system. The determinant-invariance of the transpose-operator then leads to the assertion.
\end{proof}
\begin{remark}
	The theorem above also provides us with hyperbolicy for the adjoint system. However, the linearity would essentially enable us to solve the system with less expensive methods, which could result in less degrees of freedom. We furthermore highlight that Theorem \ref{Theo:AdjointFlux} provides us with stability of the numerical scheme for the adjoint equations as well, e.g. if we have chosen the time steps in accordance with the CFL-condition for explicit time-integration in the forward problem.
\end{remark}

Updating the finite element mesh in each iteration is done via the solution $\vec{W}:\Om\rightarrow\R^2$ of the linear elasticity equation \cite{Schulz2016}
\begin{equation}
\begin{aligned}
\int_\Om\sigma(\vec{W}):\eps(\Vv)\diff x&=DJ(\Om)[\Vv] \hspace{1cm} \quad &\forall\Vv\in H_0^1(\Om,\R^2)\\
\sigma:&=\lambda_{elas} Tr(\eps(\vec{W}))I+2\mu_{elas}\eps(\vec{W})\\
\eps(\vec{W}):&=\frac{1}{2}(\nab \vec{W}+\nab \vec{W}^T)\\
\eps(\Vv):&=\frac{1}{2}(\nab\Vv+\nab\Vv^T)\text{,}
\end{aligned}
\label{Eq:29LinearElasticity}
\end{equation}
where $\sigma$ and $\eps$ are called strain and stress tensor and $\lambda_{elas}$ and $\mu_{elas}$ are called Lamé parameters. In our calculations we have chosen $\lambda_{elas}=0$ and $\mu_{elas}$ as the solution of the following Poisson problem
\begin{equation}
\begin{aligned}
-\bigtriangleup\mu&=0 \hspace{1cm} &&\text{in }&& \Om&&\\
\mu&=\mu_{max} \hspace{1cm} &&\text{on }&& \Ge&&\\
\mu&=\mu_{min} \hspace{1cm} &&\text{on }&& \Ga, \Gb\text{.}&&
\end{aligned}
\label{Eq:33Lame}
\end{equation}
The source term $DJ(\Om)[\Vv]$ in (\ref{Eq:29LinearElasticity}) consists of a volume and surface part, i.e. $DJ(\Om)[\Vv]=DJ_\Om[\Vv] + DJ_{\Ge}[\Vv]$. 
Here the volumetric share comes from our SWE shape derivative w.r.t. the first two objectives and the penalty on the volume, where we only assemble for test vector fields whose support intersects with the interface $\Ge$ and is set to zero for all other basis vector fields \cite{Welker2017}. 
The surface part comes from the parameter regularization and the minimum thinness penalty (\ref{Eq:1Thick}), where we have implemented the numerical attractive equivalent formulations
\begin{align}
DJ_4(\Om)[\Vv]=\nu_4\int_{\Ge}\Big[\nab\cdot \Vv-\langle \frac{\partial \Vv}{\partial \nv},\nv\rangle\Big]\diff s
\end{align}
and
\begin{equation}
\begin{aligned}
&DJ_5(\Om)[\Vv]&=\nu_5\int_{\Ge}\int_0^{d_{min}}\Big[&\Vv\cdot\Big\{\nab(d_{\Om}\left(x_m\right)^+)^2)-\langle\nab(d_{\Om}\left(x_m\right)^+)^2),\nv\rangle\nv\Big\}&\\
&&&+(d_{\Om}\left(x_m\right)^+)^2\Big\{\nab\cdot \Vv-\langle \frac{\partial \Vv}{\partial\nv},\nv\rangle\Big\}&\\
&&&+\Vv\cdot\vec{n}\Big\{2d_{\Om}(x_m)^+\nab d_\Om(x_m)\cdot\nv\Big\}&\\
&&&-\Vv(p_{\partial\Om}(x_m))\cdot\vec{n}(p_{\partial\Om}(x_m))2(d_{\Om}(x_m))^+\Big]\diff\xi \diff s.&
\end{aligned}
\end{equation}
In order to guarantee the attainment of useful shapes, which minimize the objective, a backtracking line search is used, which limits the step size in case the shape space is left \cite{Welker2017}, i.e. having intersecting line segments or in the case of a non-decreasing objective evaluation. As described in the algorithm before, the iteration is finally stopped if the norm of the shape derivative has become sufficiently small.

\subsection{Ex.1: The Half-Circled Mesh}\label{sec:exhalfcircled}
In the first example, we will look at the model problem - the half circle that was described in Section \ref{sec:PF}. The associated mesh is displayed in Figure \ref{fig:SWEInitHalfCircle} and was created using the finite element mesh generator GMSH \cite{Geuzaine2009}, we have meshed finer around the obstacle to ensure a high resolution. We set Gaussian initial conditions as $\hat{U}_0=(1+\exp(-15x^2 - 15(y-1)^2),0,0)$, which result in a wave travelling in time towards the boundaries.  As before, we interpret $\Ga,\Gd,\Ge$ as coastline, open sea and obstacle boundary. Accordingly, we prescribe the boundary conditions using rigid-wall conditions on $\Ga,\Ge$ and outflow boundaries on $\Gd$. The parameters in the shallow water system are set as follows: For the weight of the diffusion terms in the momentum equation we set $\mu_f=(0.01,0.01)$ and determine $\mu_v$ by the usage of the mentioned shock detector \cite{Persson2006}. The gravitational acceleration is fixed at roughly $9.81$ and the parameter $K$ in Manning's formula is at $0.049$ for a sandy beach. Our calculations are performed for two test cases - a linear decreasing bottom $z=0.5-0.25y$ and a non-flat bottom determined by a Gaussian peak $z=\exp(-6(x-0.5)^2 - 6(y-0.2)^2)$, as displayed in Figure \ref{fig:SWEInitHalfCircle}.  We are targeting a minimal mechanical wave energy for waves above the water's rest height, such that the energy and sigmoid function are defined in terms of $H+z$ for threshold $H_{\text{cr}}=1$ and slope parameter $\alpha=10$ such as zeroed velocities by setting $\nu_1=\nu_2=1$. In addition, we penalize volume and thinness by setting $\nu_3=\expnumber{1}{-4},\nu_5=\expnumber{1}{-2}$ such as enforcing a stronger regularization by $\nu_4=\expnumber{1}{-4}$. 
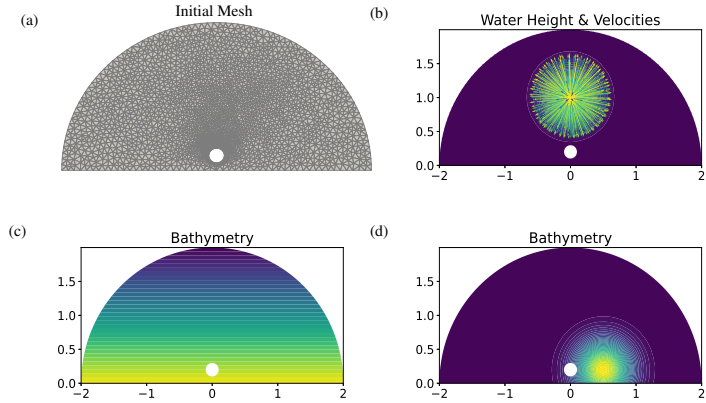
\begin{figure}[!htb]
	\centering
	\begin{tikzpicture}
	\node[anchor=south west,inner sep=0] (1) {\vspace{-0.9cm}
		\includegraphics[width=4.5cm,height=3.0cm]{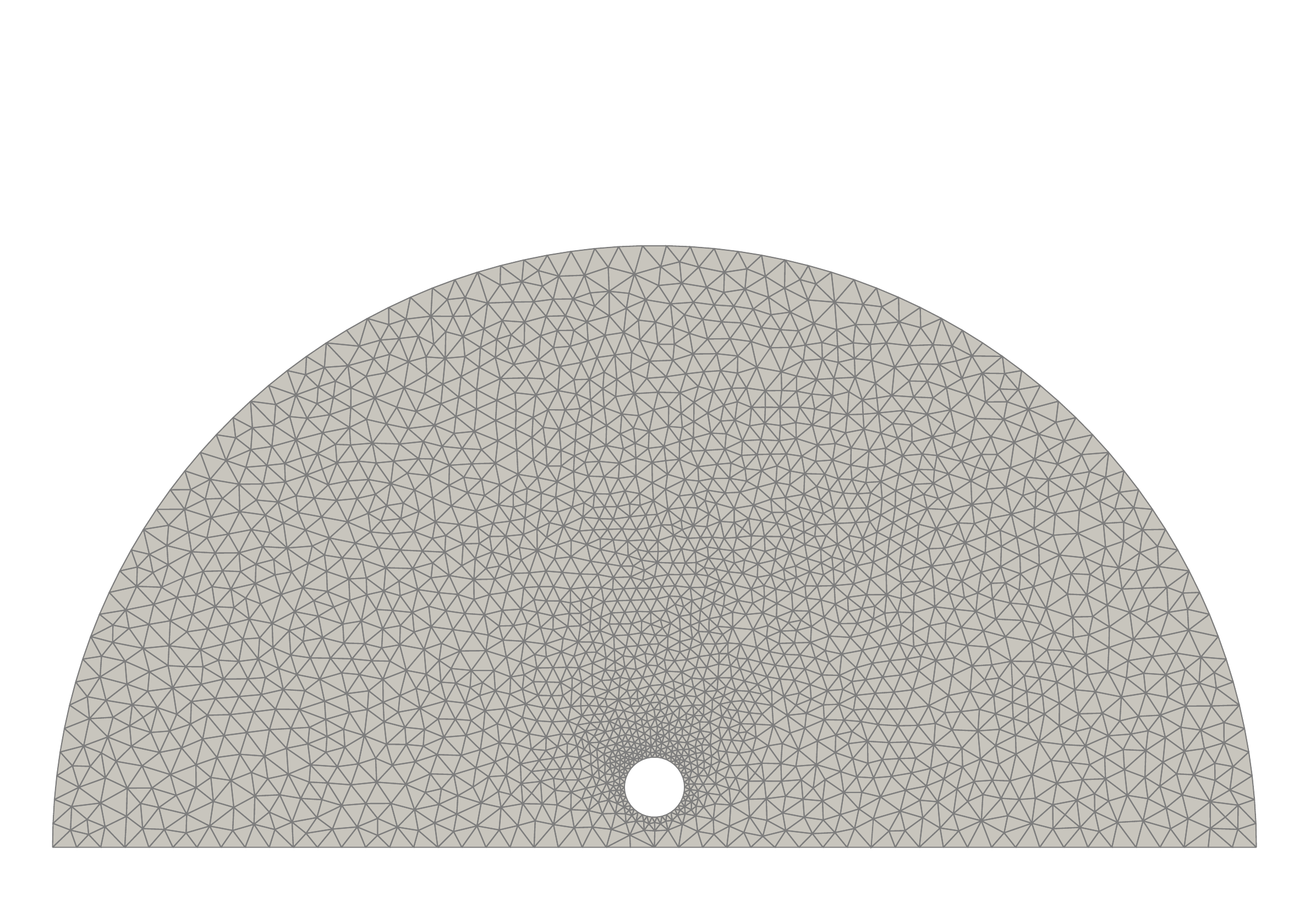}};
	\mbox{}\vspace{2cm}\node[below right = -3.1cm and 0.001cm of 1](2)
	{\mbox{}\vspace{10cm}\includegraphics[width=4.5cm,height=3.5cm,scale=0.5]{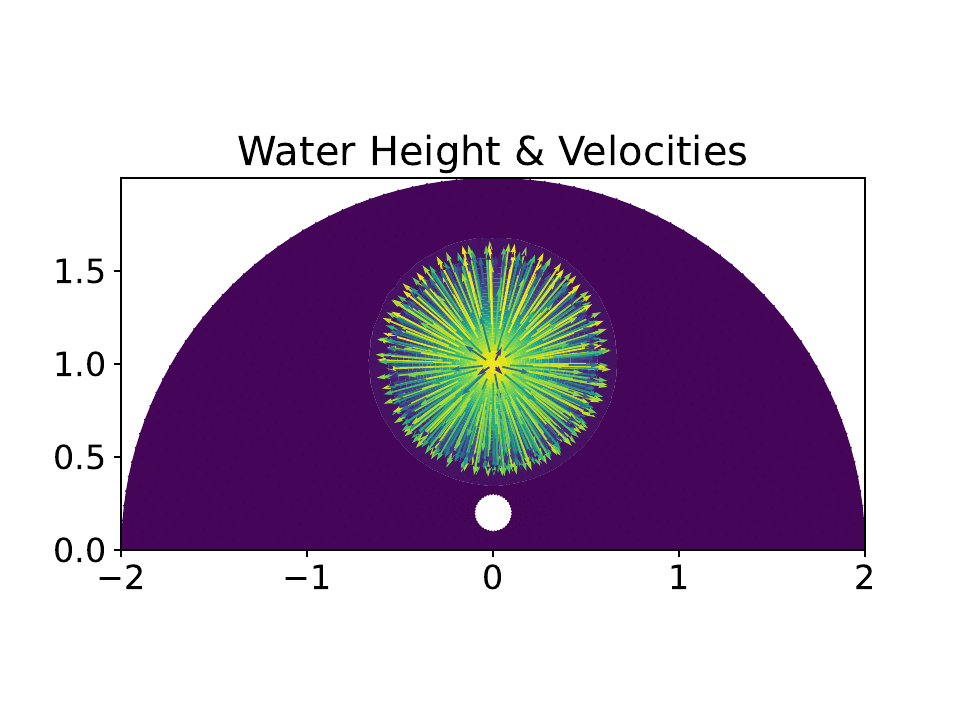}};
	\node[below= -0.85cm of 2](3)
	{\includegraphics[width=4.5cm,height=3.5cm,scale=0.5]{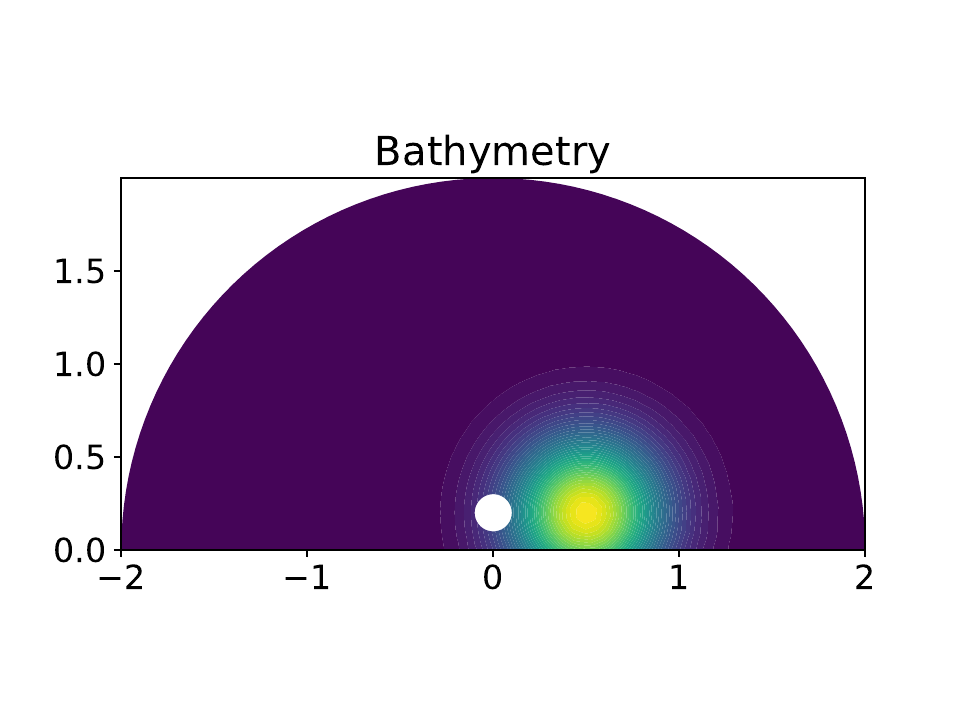}};
	\node[left= 0.015cm of 3](4)
	{\includegraphics[width=4.5cm,height=3.5cm,scale=0.5]{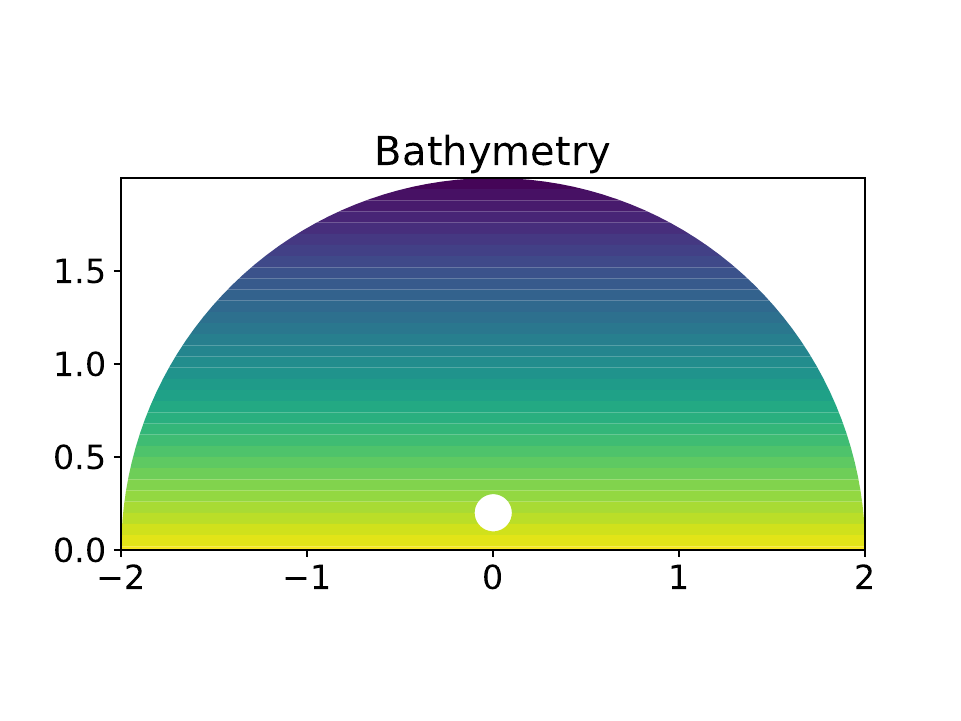}};
	\node[above= -0.85cm of 1] (7) 
	{\tiny Initial Mesh};
			\node[above left = -1cm and -0.09cm of 1](8)
	{\tiny (a)};
	\node[above left = -1cm and -0.14cm of 2](9)
	{\tiny (b)};
	\node[above left = -1cm and -0.09cm of 4](10)
	{\tiny (c)};
	\node[above left = -1cm and -0.14cm of 3](11)
	{\tiny (d)};
	\end{tikzpicture}
	\caption[Ex.1 Mesh, Water Height, Velocities \& Sediment]{(a) Initial Mesh and Obstacle, (b) Field State at $t=0.1$, (c) Linear Bathymetry, (d) Gaussian Peak Bathymetry}
	\label{fig:SWEInitHalfCircle}
\end{figure} In this example we have used an implicit backward Euler time-scheme and a DG-method of first order that was described before. For the spatial discretization, we have used the HLLE-flux function for the convective terms and $C_{IP}=20$ in the SIPG method. Solving the state equations requires the definition of the time-horizon, e.g. as $\tilde{T}=(0,T)=(0,2.5)$, which is chosen to include one full wave period, i.e. the travel of a wave to and from the shore. The discretization in time is based on a step size of $\diff t=\expnumber{5}{-3}$. Due to the nonlinear nature of the SWE we have used a Newton solver, where we set the absolute and relative tolerance as $\eps_{abs}=\eps_{rel}=\expnumber{1}{-6}$. The solution of the adjoint problem follows likewise, but stepping backwards in time. Since the problem is linear, a Newton solver is no longer needed. Having solved state and adjoint equations the mesh deformation is performed as described, where we specify $\mu_{min}=10$ and $\mu_{max}=100$ in (\ref{Eq:33Lame}). The step size is at $\rho=1$ and shrinks whenever criteria for line searches are not met. In Figure \ref{fig:OptiHalfCircleSWE} results of the shape optimization are displayed, firstly for a linear and secondly a Gaussian bottom after $44$ and $33$ steps of optimization.
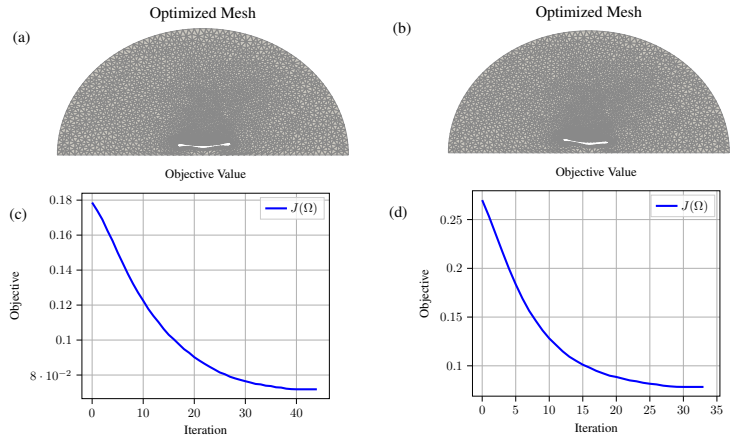
\begin{figure}[htb!]
	\centering
	\begin{tikzpicture}
	\node[anchor=south west,inner sep=0] (1) {
		\includegraphics[width=4.5cm,height=3.0cm]{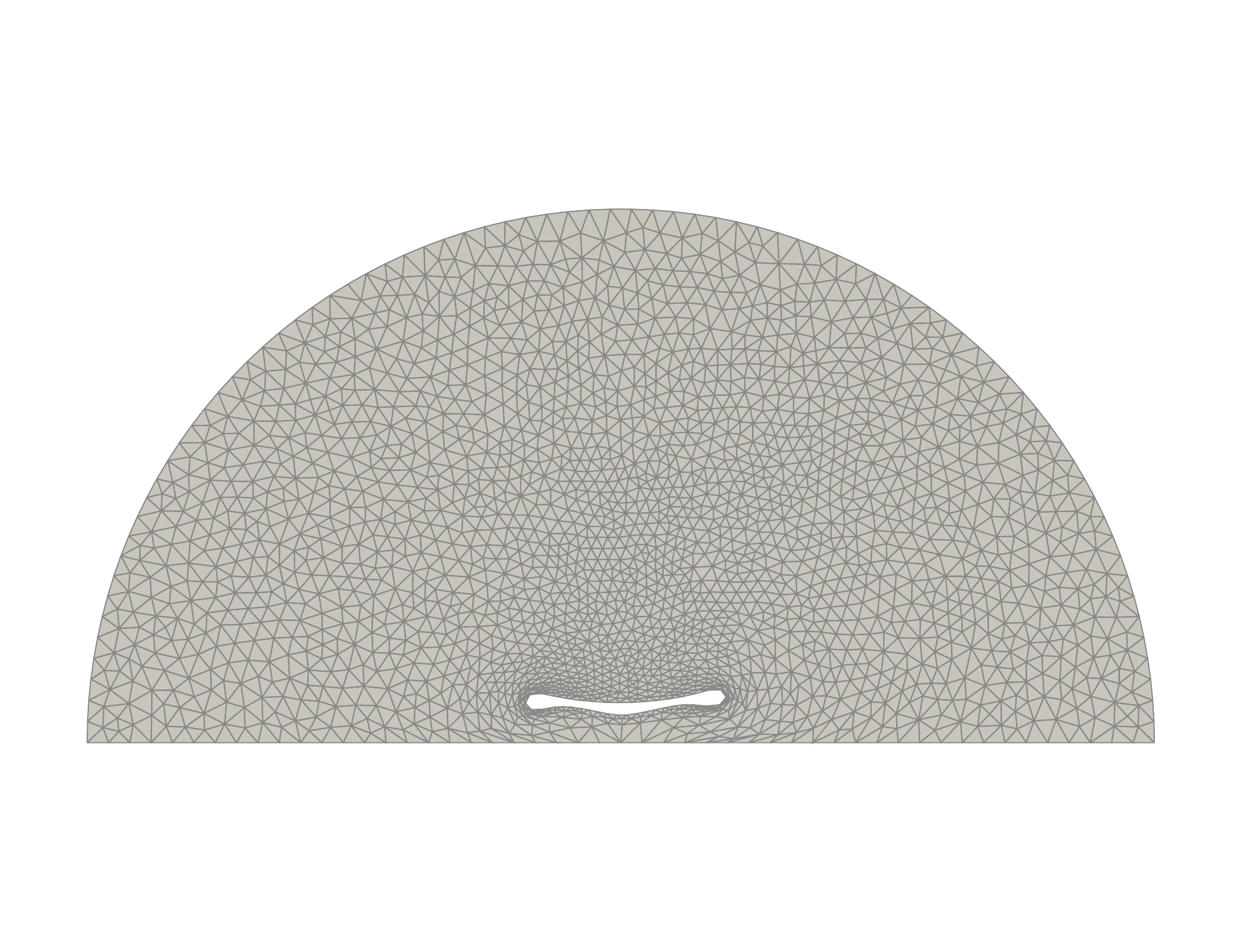}};
	\node[right= 0.5cm of 1](2)
	{\includegraphics[width=4.5cm,height=3.0cm]{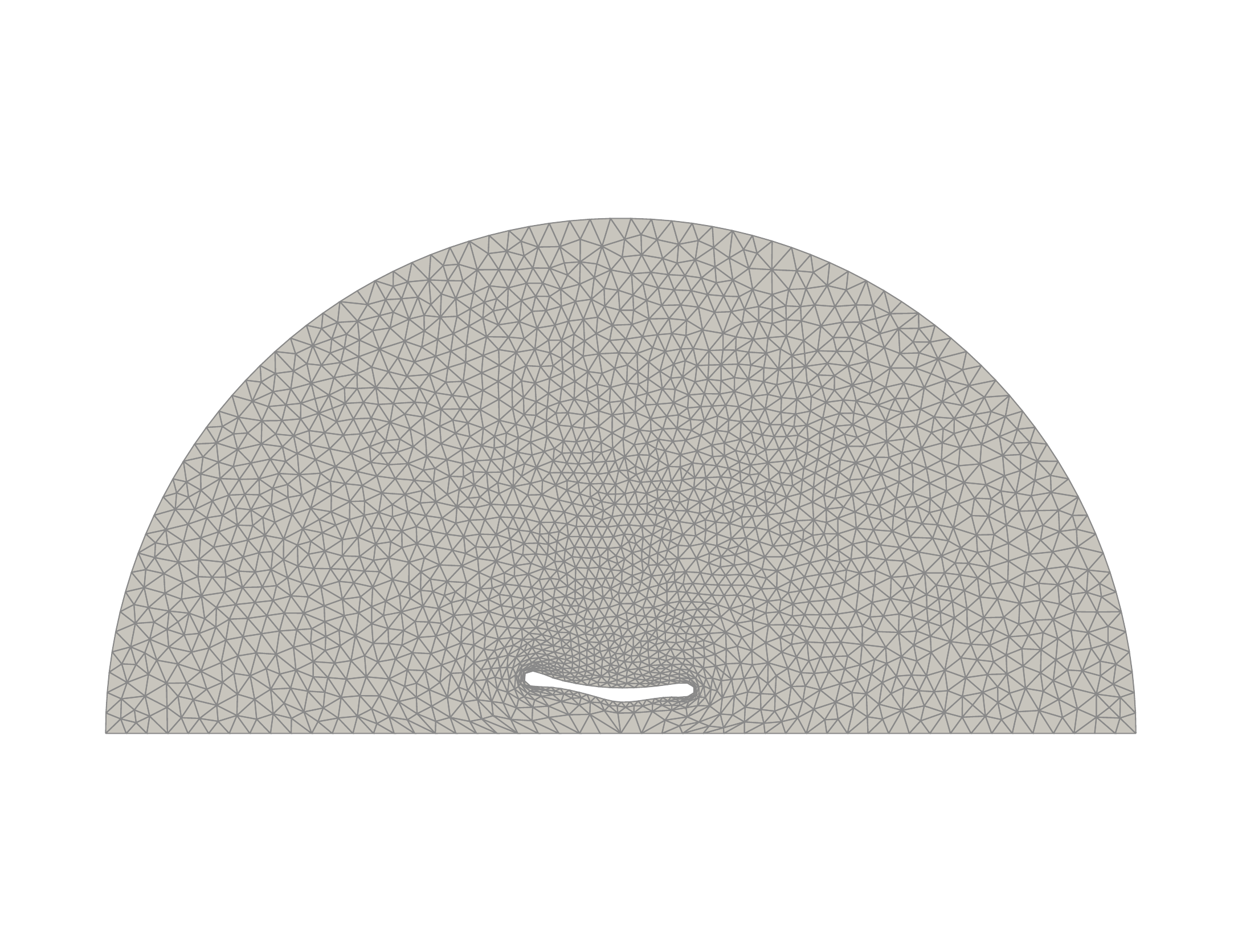}};
	\node[above= -0.7cm of 1] (7) 
	{\tiny Optimized Mesh};
	\node[above= -0.8cm of 2] (8)
	{\tiny Optimized Mesh};
	\node[below right = -0.65cm and -5cm of 1](3)
	{\scalebox{0.48}{
\begin{tikzpicture}

\begin{axis}[
legend cell align={left},
legend style={fill opacity=0.8, draw opacity=1, text opacity=1, draw=white!80!black},
tick align=outside,
tick pos=left,
title={Objective Value},
x grid style={white!69.0196078431373!black},
xlabel={Iteration},
xmajorgrids,
xmin=-2, xmax=46.45,
xtick style={color=black},
y grid style={white!69.0196078431373!black},
ylabel={Objective},
ymajorgrids,
ymin=0.0665101746852583, ymax=0.183949141309512,
ytick style={color=black}
]
\addplot [line width=1.64pt, blue]
table {%
0 0.178611006462955
1 0.174014277495061
2 0.168942850018294
3 0.162556239613287
4 0.156758351337907
5 0.150037021712602
6 0.143965486097876
7 0.137833915572868
8 0.132230772942651
9 0.127104977074808
10 0.122555673075063
11 0.117782523230384
12 0.113799018674096
13 0.110246574878673
14 0.106367249065556
15 0.103008871210741
16 0.100428922854886
17 0.0975933051526757
18 0.0947879763917409
19 0.0927173707583259
20 0.0902149247681994
21 0.0882768641950423
22 0.0864547584873849
23 0.0847414282811796
24 0.0831306096655933
25 0.0814742235341243
26 0.0804477512040719
27 0.0790236986057645
28 0.0780732935719587
29 0.0772448849783594
30 0.0764697191776202
31 0.075738895299049
32 0.0749110956073338
33 0.0746584838706872
34 0.0739139268951925
35 0.0736932105728908
36 0.0730016976498697
37 0.0728359792245216
38 0.072208117596665
39 0.0719511103923125
40 0.0718483095318152
41 0.0718472041325100
42 0.0718472032623300
43 0.0718471155586301
44 0.0718471145563111
};
\addlegendentry{$J(\Omega)$}
\end{axis}

\end{tikzpicture}}};
	\node[below right = -0.8cm and -4.8cm of 2](4)
	{\scalebox{0.48}{
\begin{tikzpicture}

\begin{axis}[
legend cell align={left},
legend style={fill opacity=0.8, draw opacity=1, text opacity=1, draw=white!80!black},
tick align=outside,
tick pos=left,
title={Objective Value},
x grid style={white!69.0196078431373!black},
xlabel={Iteration},
xmajorgrids,
xmin=-1.45, xmax=35.45,
xtick style={color=black},
y grid style={white!69.0196078431373!black},
ylabel={Objective},
ymajorgrids,
ymin=0.0688665028448153, ymax=0.279579521117153,
ytick style={color=black}
]
\addplot [line width=1.64pt, blue]
table {%
0 0.270001656650229
1 0.25351078847339
2 0.235324160400768
3 0.217194508962811
4 0.199576568478456
5 0.183592865625892
6 0.169121436480127
7 0.156458344556419
8 0.146125983203896
9 0.13636357163698
10 0.128115677177356
11 0.121454287922772
12 0.114844989078713
13 0.109317907043497
14 0.105115684192682
15 0.101141220725366
16 0.0981835788086664
17 0.094865935617389
18 0.0923326055002762
19 0.0898769609659449
20 0.0885741702744451
21 0.0868088444846374
22 0.0850522985253952
23 0.0841175377160838
24 0.0825832078805201
25 0.0815996468836118
26 0.0809112810505122
27 0.0796334449065232
28 0.078945069367712
29 0.0784443673117398
30 0.0783423271106300
31 0.0783422121472321
32 0.0783422025418520
33 0.0783422011121591
};
\addlegendentry{$J(\Omega)$}
\end{axis}

\end{tikzpicture}}};
				\node[above left = -1cm and -0.09cm of 1](8)
	{\tiny (a)};
	\node[above left = -1cm and -0.14cm of 2](9)
	{\tiny (b)};
	\node[above left = -1cm and -0.53cm of 3](10)
	{\tiny (c)};
	\node[above left = -1cm and -0.14cm of 4](11)
	{\tiny (d)};
	\end{tikzpicture}
	\caption[Ex.1 Optimization Results]{(a) Optimized Obstacle for Linear Seabed, (b) Optimized Obstacle for Gaussian Seabed, (c) Objective for Linear Seabed, (d) Objective for Gaussian Seabed}
	\label{fig:OptiHalfCircleSWE}
\end{figure}

 The deformations are symmetric in the first and in the opposing direction of the sediment hill in the second case. As we observe in the lower part of Figure \ref{fig:OptiHalfCircleSWE}, we have achieved notable decreases in the objective.

\subsection{Ex.2: Langue de Barbarie}\label{sec:exldb}
A more realistic computation is performed in the second example. Here we look at the LdB a coastal section in the north of Dakar, Senegal. In 1990 it consisted of a long offshore island, which eroded in three parts within two decades. Waves now travel unhindered to the mainlands, which causes severe damage and already destroyed large habitats.
Adjusting our model to this specific coastal section starts on mesh level. Shorelines are taken from the free GSHHG\footnote[1]{https://www.ngdc.noaa.gov/mgg/shorelines/} databank, following \cite{Avdis2016}. We build up an interface from a geographical information system (QGIS3) for processing the data to a computer aided design software (GMSH) for the mesh generation.  Similar to the preceding example, we interpret $\Ga$ as coastline of the mainland, $\Gd$ as the open sea boundary such as  $\Ge$ as the three offshore islands (cf. to Figure \ref{fig:WavePropagationLDB},\ref{fig:ShapeOptLDB}).
\begin{figure}[!htbp]
	\centering
		\begin{tikzpicture}
	\node[anchor=north west,inner sep=0] (6)  {\includegraphics[scale=0.3]{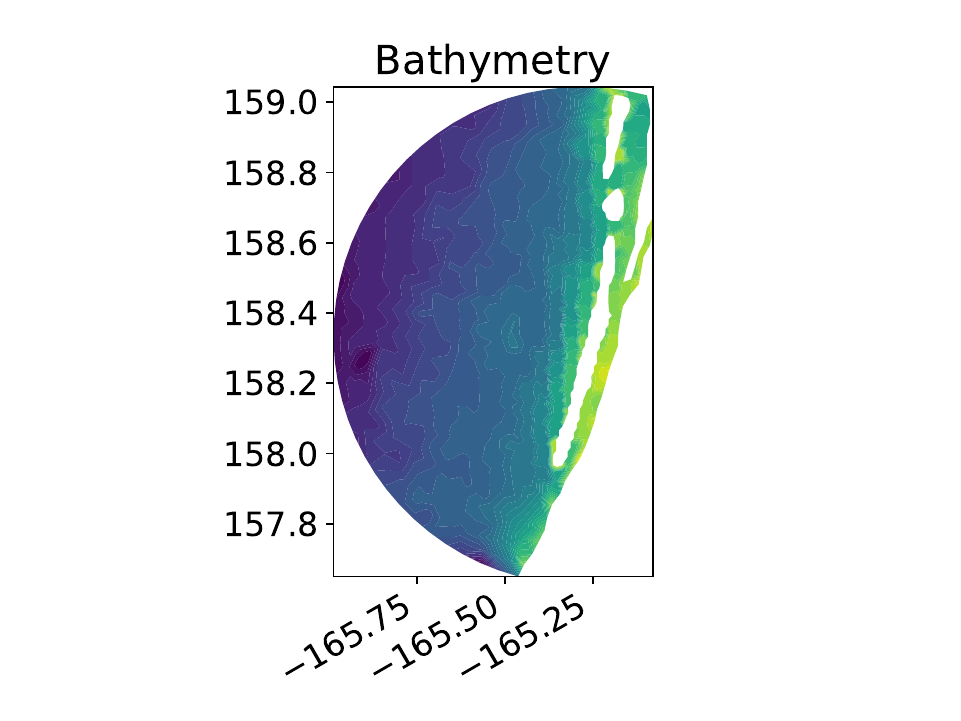}};
	\end{tikzpicture}
	\caption[Ex.2 LdB Sediment]{LdB Sediment Elevation}
	\label{fig:SedimentLDB}
\end{figure}

\begin{figure}[htb!]
	\centering
	\begin{subfigure}{0.2\textwidth}
		\centering
			\begin{tikzpicture}
		\node[anchor=south west,inner sep=0] (1) 
		{\includegraphics[scale=0.3]{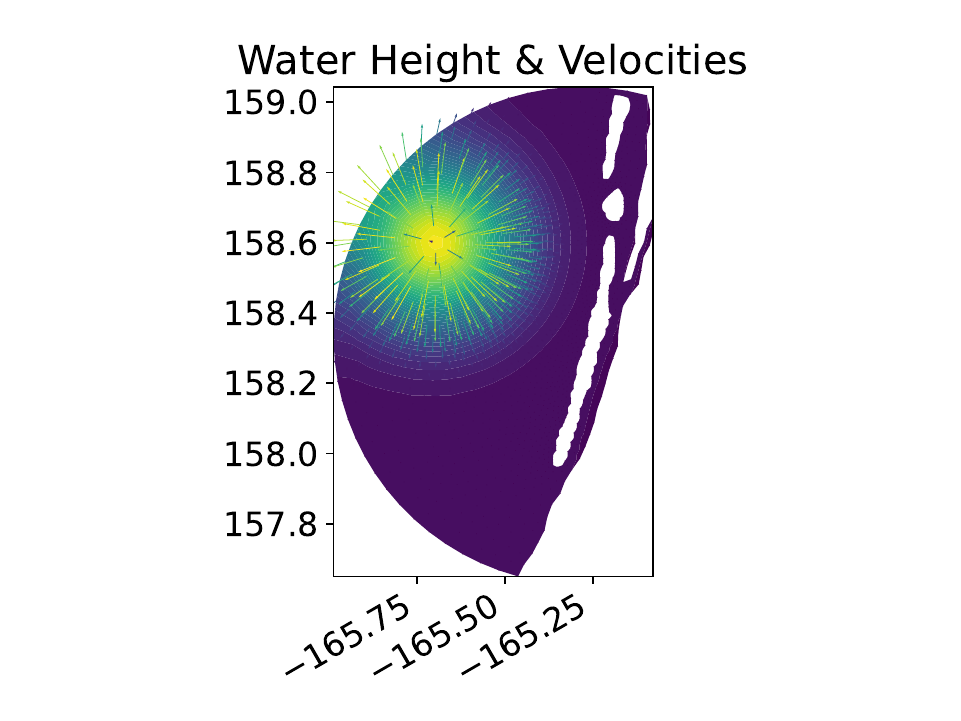}};
		\end{tikzpicture}
		\caption{$t=0$}
	\end{subfigure}
	\hfill
	\begin{subfigure}{0.2\textwidth}
		\centering
			\begin{tikzpicture}
		\node[anchor=south west,inner sep=0] (1) 
		{\includegraphics[scale=0.3]{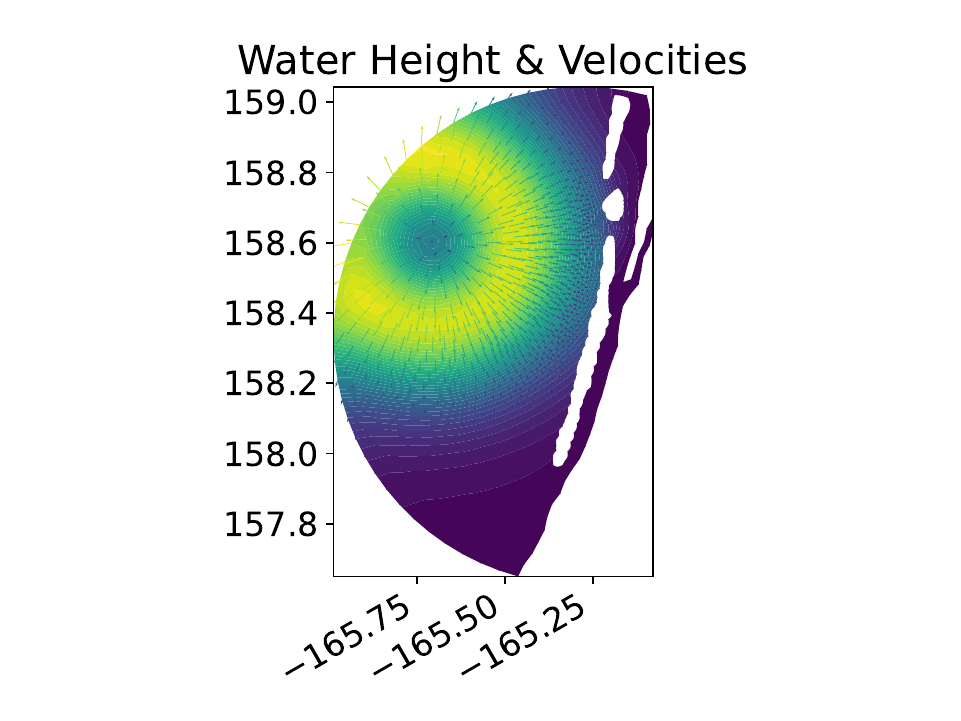}};
		
		\end{tikzpicture}
		\caption{$t=0.5$}
	\end{subfigure}
	\hfill
	\begin{subfigure}{0.2\textwidth}
		\centering
			\begin{tikzpicture}
		\node[anchor=south west,inner sep=0] (1) 
		{\includegraphics[scale=0.3]{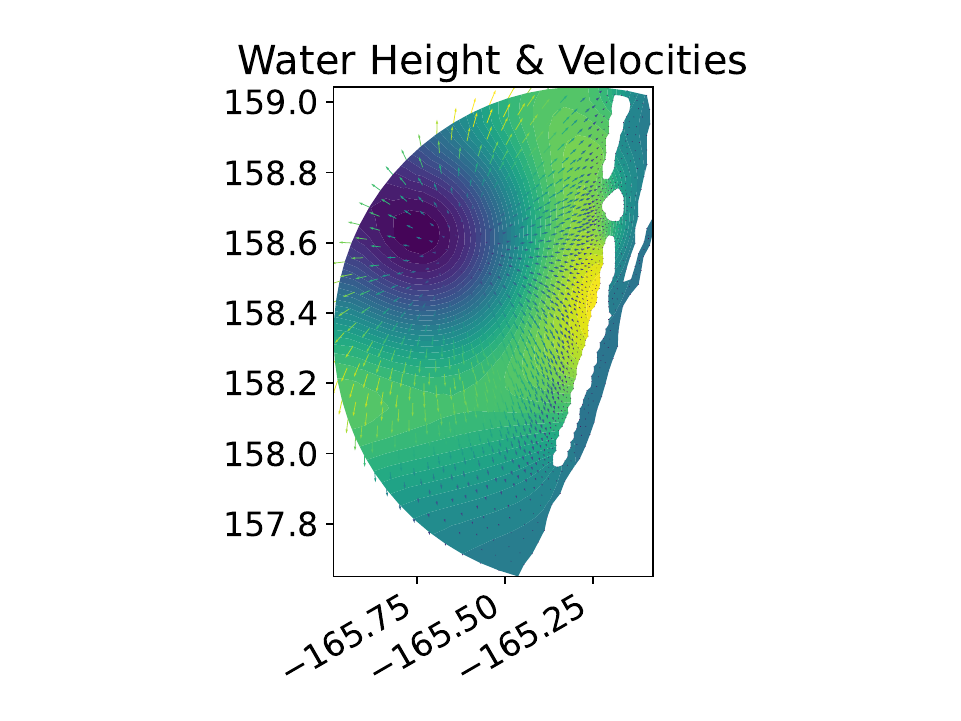}};
		\end{tikzpicture}
		\caption{$t=1$}
	\end{subfigure}
	\hfill
	\begin{subfigure}{0.2\textwidth}
		\centering
			\begin{tikzpicture}
		\node[anchor=south west,inner sep=0] (1) 
		{\includegraphics[scale=0.3]{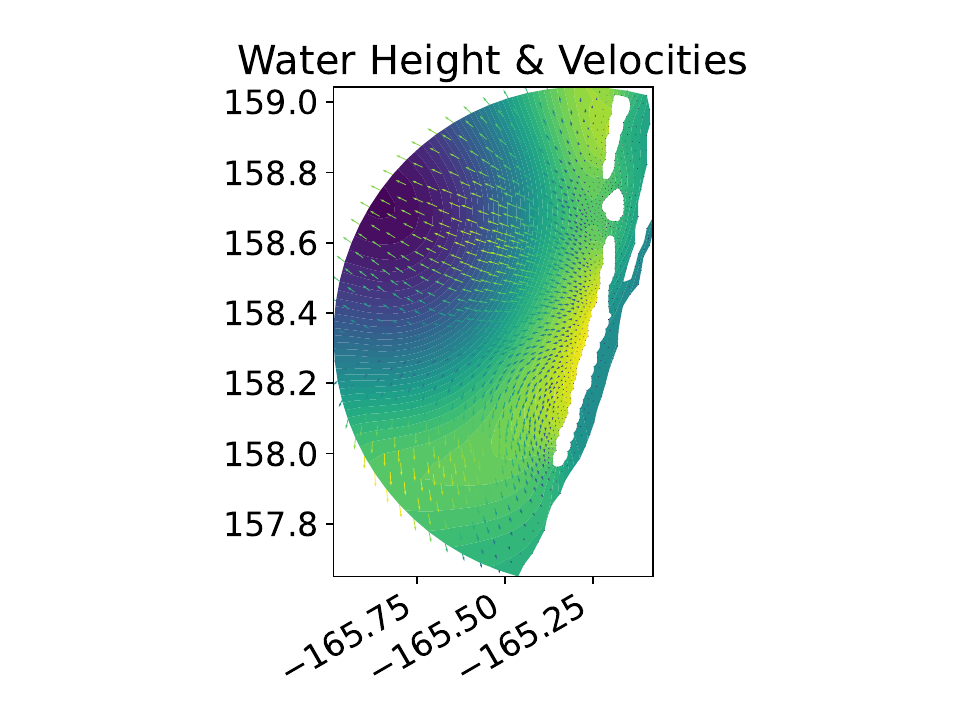}};
		\end{tikzpicture}
		\caption{$t=1.5$}
	\end{subfigure}
	\caption[Ex.2 Wave Propagation]{Visualization of a Wave Described by Height and Velocities, Travelling Towards the Shore for Initial Obstacle.}
	\label{fig:WavePropagationLDB}
\end{figure}

 As before, we start with Gaussian initial conditions for the height of the water. Sediment data is taken from the GEBCO\footnote[2]{https://www.gebco.net/} databank, where bathymetric elevation is mapped to a mesh point using a nearest neighbors algorithm. The sediment elevation can be taken from Figure \ref{fig:SedimentLDB}, while the wave propagation can be extracted from Figure \ref{fig:WavePropagationLDB}. The remaining model-settings are similar to Section \ref{sec:exhalfcircled}. Figure \ref{fig:ShapeOptLDB} pictures initial, such as deformed mesh and obstacle after $30$ steps of optimization.
\begin{figure}[htb!]
	\centering
	\begin{subfigure}{0.4\textwidth}
		\centering
		\begin{tikzpicture}
		\node[anchor=south west,inner sep=0] (1) {
		\includegraphics[width=5cm,height=3.0cm]{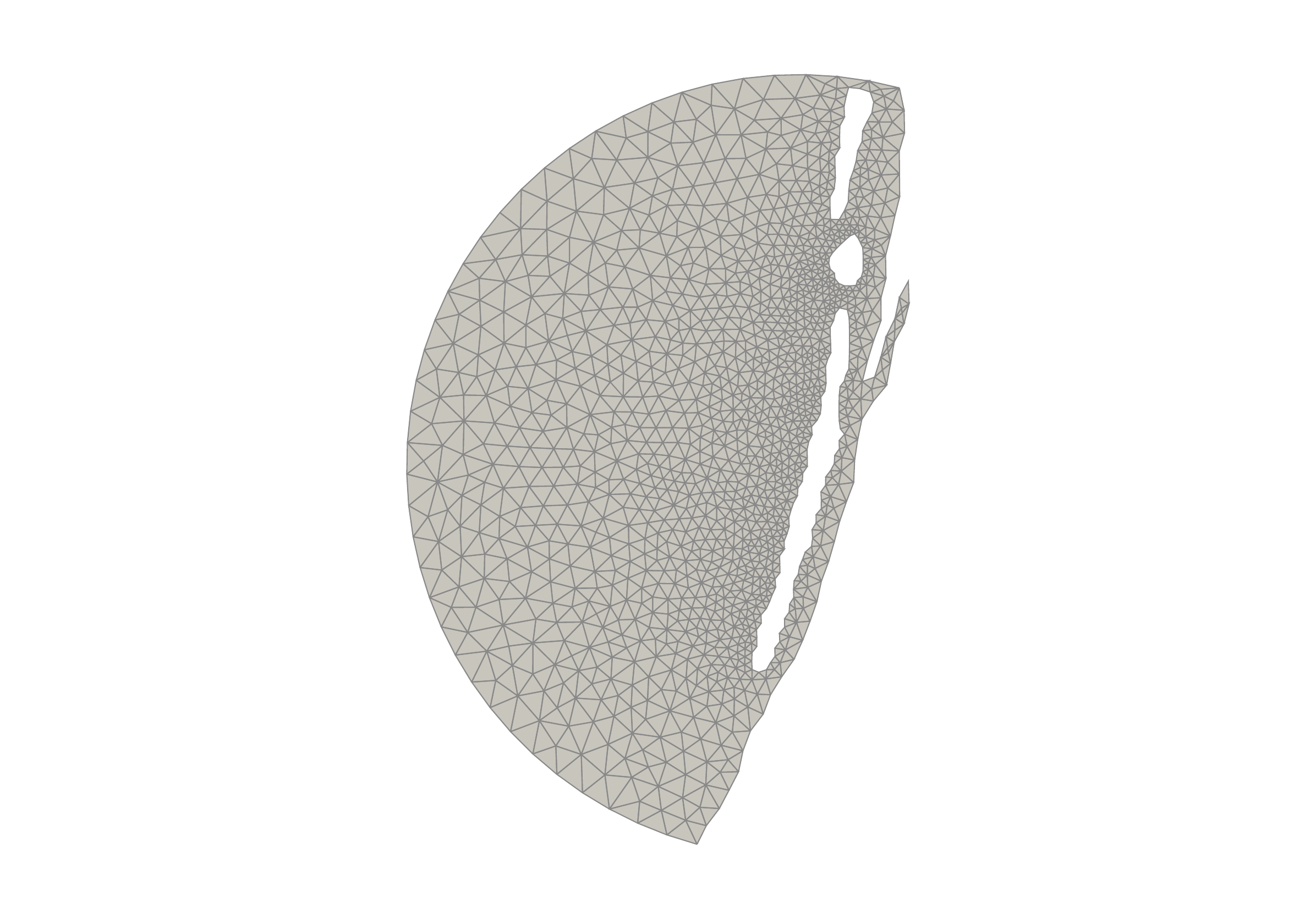}};
		\end{tikzpicture}
		\caption{Initial Mesh}
	\end{subfigure}
	\hfill
	\begin{subfigure}{0.4\textwidth}
		\centering
		\begin{tikzpicture}
		\node[anchor=south west,inner sep=0] (1)
	{\includegraphics[width=4.5cm,height=3.0cm]{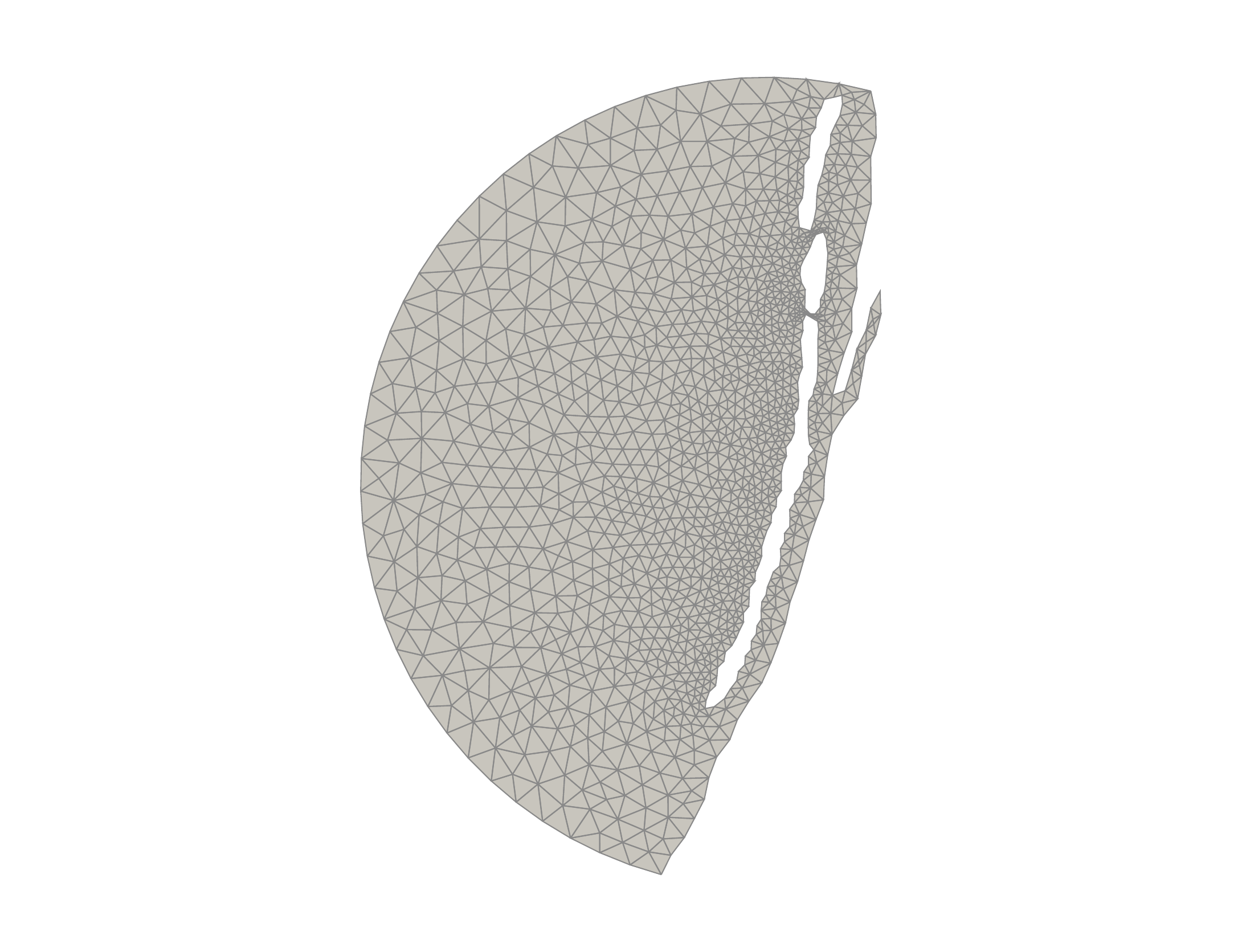}};
		\end{tikzpicture}
		\caption{Optimized Mesh}
	\end{subfigure}
	\caption[Ex.2 Initial, Optimized Mesh \& Objective]{Initial and Optimized Mesh and Obstacle}
	\label{fig:ShapeOptLDB}
\end{figure}

One can observe a similar behaviour as in Subsection \ref{sec:exhalfcircled}, where the obstacle is stretched to protect an as large as possible area. In this setting, the optimizer suggests to reconnect the three islands. However, rebuilding the complete island would either call for a remeshing procedure or an alternative algorithm for shape optimization, e.g. level sets as in \cite{Keuthen2015} are capable of similar. We highlight that obtained results must be treated with caution, since rebuilding would require an excessive amount of landmass. As an alternative, simulations with artificial offshore islands subject to volume constraints can be performed.  In Figure \ref{fig:TargetFunctionLDB} the convergence of the objective can be observed. 
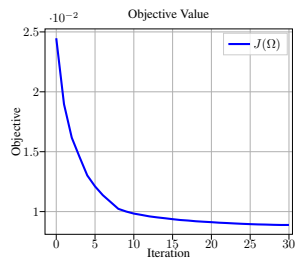
\begin{figure}[!htb]
	\centering
	\begin{tikzpicture}
	\node[anchor=north west,inner sep=0] (6)  {\scalebox{0.48}{
\begin{tikzpicture}

\begin{axis}[
legend cell align={left},
legend style={fill opacity=0.8, draw opacity=1, text opacity=1, draw=white!80!black},
tick align=outside,
tick pos=left,
title={Objective Value},
x grid style={white!69.0196078431373!black},
xlabel={Iteration},
xmajorgrids,
xmin=-1.45, xmax=30.45,
xtick style={color=black},
y grid style={white!69.0196078431373!black},
ylabel={Objective},
ymajorgrids,
ymin=0.00810317657677666, ymax=0.0252512247910592,
ytick style={color=black}
]
\addplot [line width=1.64pt, blue]
table {%
0 0.0244717680540464
1 0.0189625341585475
2 0.0161690074293594
3 0.0145629711262295
4 0.0130181427426235
5 0.0121032088595392
6 0.0113724865830049
7 0.0108034769738538
8 0.0102305565968628
9 0.0100081861431565
10 0.00983410650064072
11 0.00972085931019617
12 0.00960076074877076
13 0.00951741957137013
14 0.00944562798221293
15 0.00936672694220101
16 0.00929727616988194
17 0.00925269280982534
18 0.00919296459448541
19 0.00915325586172253
20 0.00911750545525922
21 0.00907240608338492
22 0.00904413502323294
23 0.00900930922005377
24 0.00897879020513196
25 0.00895303817809148
26 0.00892977339786163
27 0.00891641858682834
28 0.00889598664267437
29 0.0088826333137895
30 0.0088841106161288
};
\addlegendentry{$J(\Omega)$}
\end{axis}

\end{tikzpicture}}};
	\end{tikzpicture}
	\caption[Ex.2 Objectivel]{Objective for LdB Mesh}
		\label{fig:TargetFunctionLDB}
\end{figure}
\FloatBarrier

\subsection{Ex.3: World Mesh}\label{sec:exworldmesh}
In the third and last example, we extend presented techniques to immersed-manifolds, in order to perform global shore protection. For this, we define $\Om$ to be a smooth $m$-dimensional manifold immersed in $R^n$, where $m=2$ denotes the topological dimension and $n=3$ the geometric dimension. We assume a similar setting as before, where $\Ga$ represents the continent of Africa and $\Gb$ the remaining coastal points. In addition, we have placed three initial circled obstacles with boundary $\Ge$ in before the shore of West-Africa that serve as obstacle.
\begin{figure}[!htb]

	\centering
	\begin{tikzpicture}
	\node[anchor=south west,inner sep=0] (1) {
	\includegraphics[width=4.7cm,height=3.0cm]{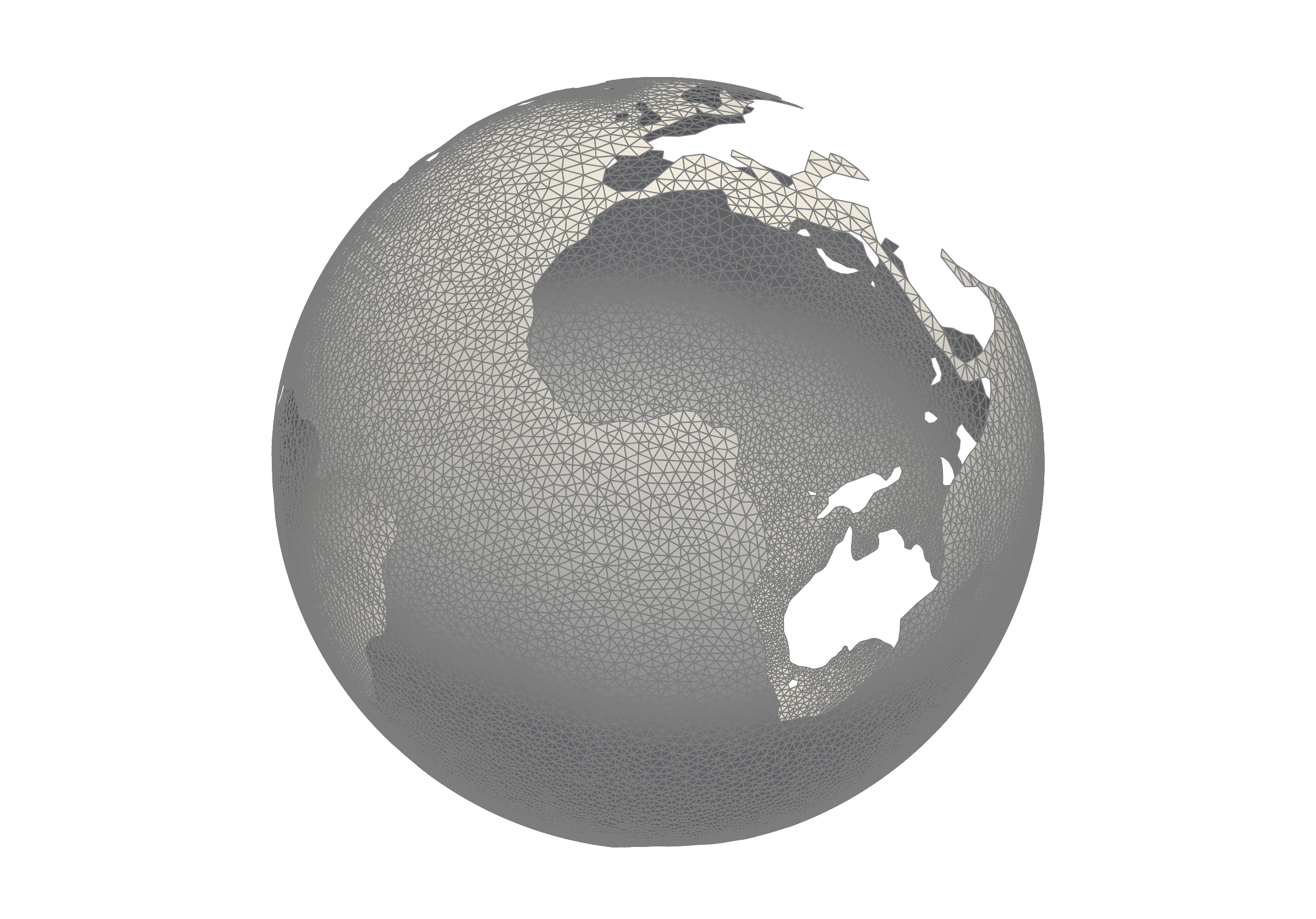}};
	\end{tikzpicture}
    
	\caption[Ex.3 High Resolution World Mesh]{High Resolution World Mesh}
	\label{fig:WorldMesh}
\end{figure}
From the implementational side we have again used the GSHHG databank to obtain coastal data and mapped the points to a PolarSphere in GMSH (cf. to Figure \ref{fig:WorldMesh}). For the discretization we follow \cite{Rognes2013}, from which an extension of the FEniCS software to the scenario above stems from. We aim for a solution in the geometric space i.e. $U_h=(H_h,u_hH_h,v_hH_h,w_hH_h)$ relying on $DG$-elements, i.e. $DG_1\times DG_3$, where we weakly enforce the vector-valued velocity to be in the spherical tangent space. Alternatively, we could solve in the mixed discrete Function Space $DG_1\times RT_1$, where $RT_1$ denotes Raviar-Thomas finite elements, which lie in the tangent space simple from its construction. We define initial conditions in the geometric space as $U_0=(2+\exp(-c(x-x_0)^2 - c(y-y_0)^2 - c(z-z_0)^2),0,0,0)$ for suitable coordinates $(x_0,y_0,z_0)$ and constant $c$. In contrast to the examples before, open sea boundaries are not required any more, such that all boundaries are subject to rigid boundary conditions. The seabed is for simplicity assumed to be flat. The remaining model-settings are similar to Subsection \ref{sec:exhalfcircled}. The wave propagation is visualized in Figure \ref{fig:WavePropagationSphere}.
\begin{figure}[htb!]
	\centering
	\begin{subfigure}{0.2\textwidth}
		\centering
			\begin{tikzpicture}
		\node[anchor=south west,inner sep=0] (1) 
		{\includegraphics[scale=0.11]{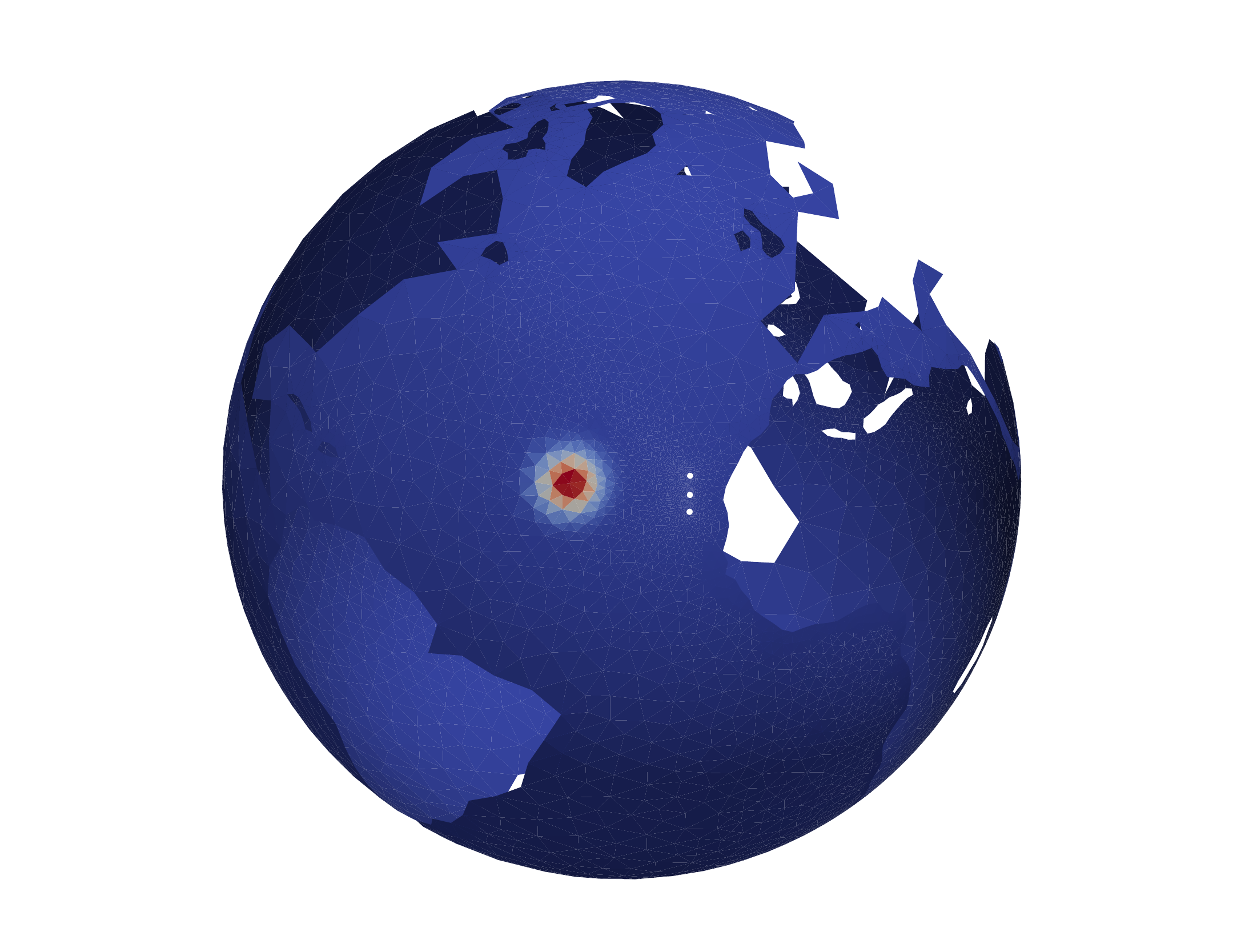}};
		\end{tikzpicture}
		\caption{$t=0$}
	\end{subfigure}
	\hfill
	\begin{subfigure}{0.2\textwidth}
		\centering
			\begin{tikzpicture}
		\node[anchor=south west,inner sep=0] (1) 
		{\includegraphics[scale=0.11]{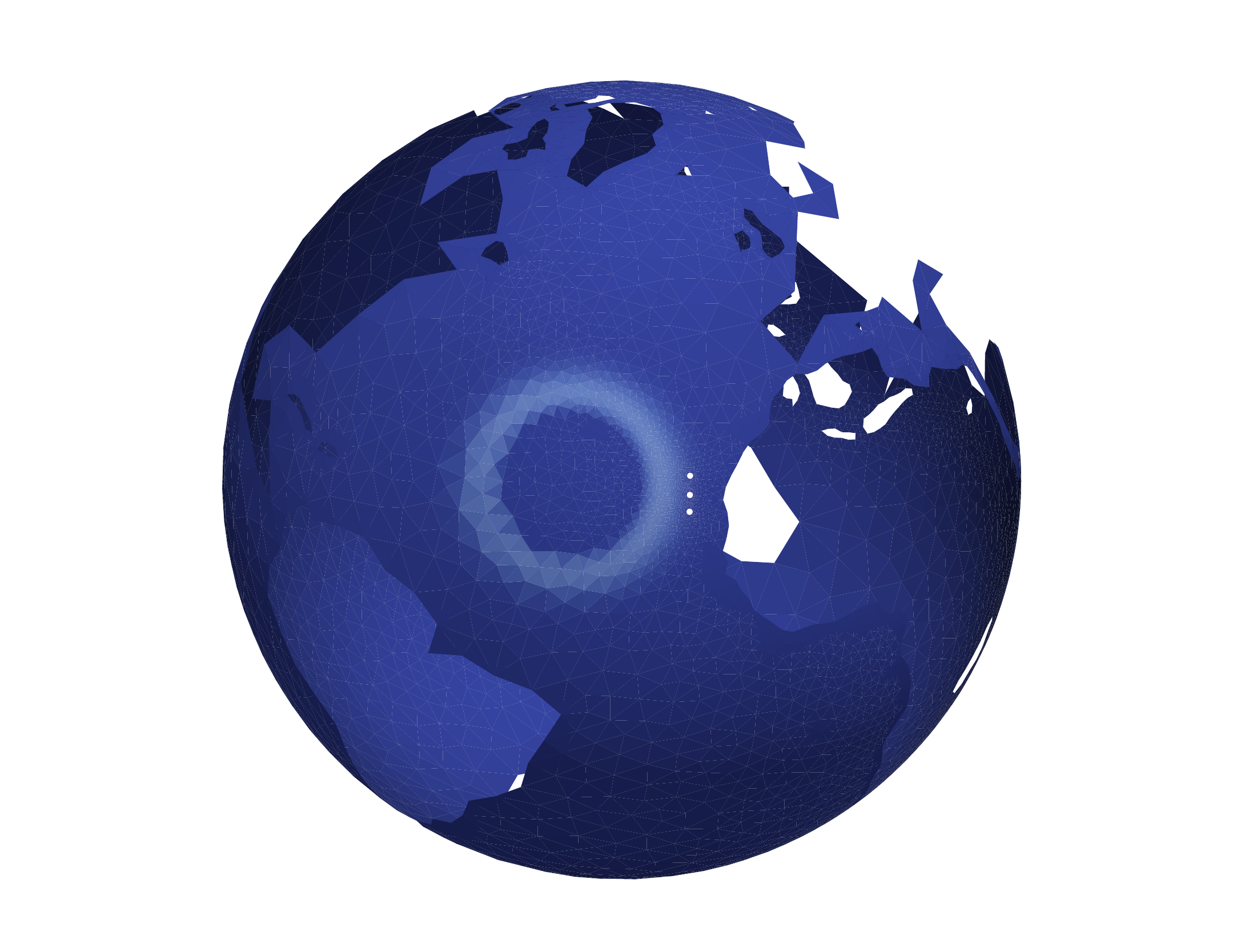}};
		\end{tikzpicture}
		\caption{$t=0.5$}
	\end{subfigure}
\hfill
	\begin{subfigure}{0.2\textwidth}
	\centering
		\begin{tikzpicture}
	\node[anchor=south west,inner sep=0] (1) 
	{\includegraphics[scale=0.11]{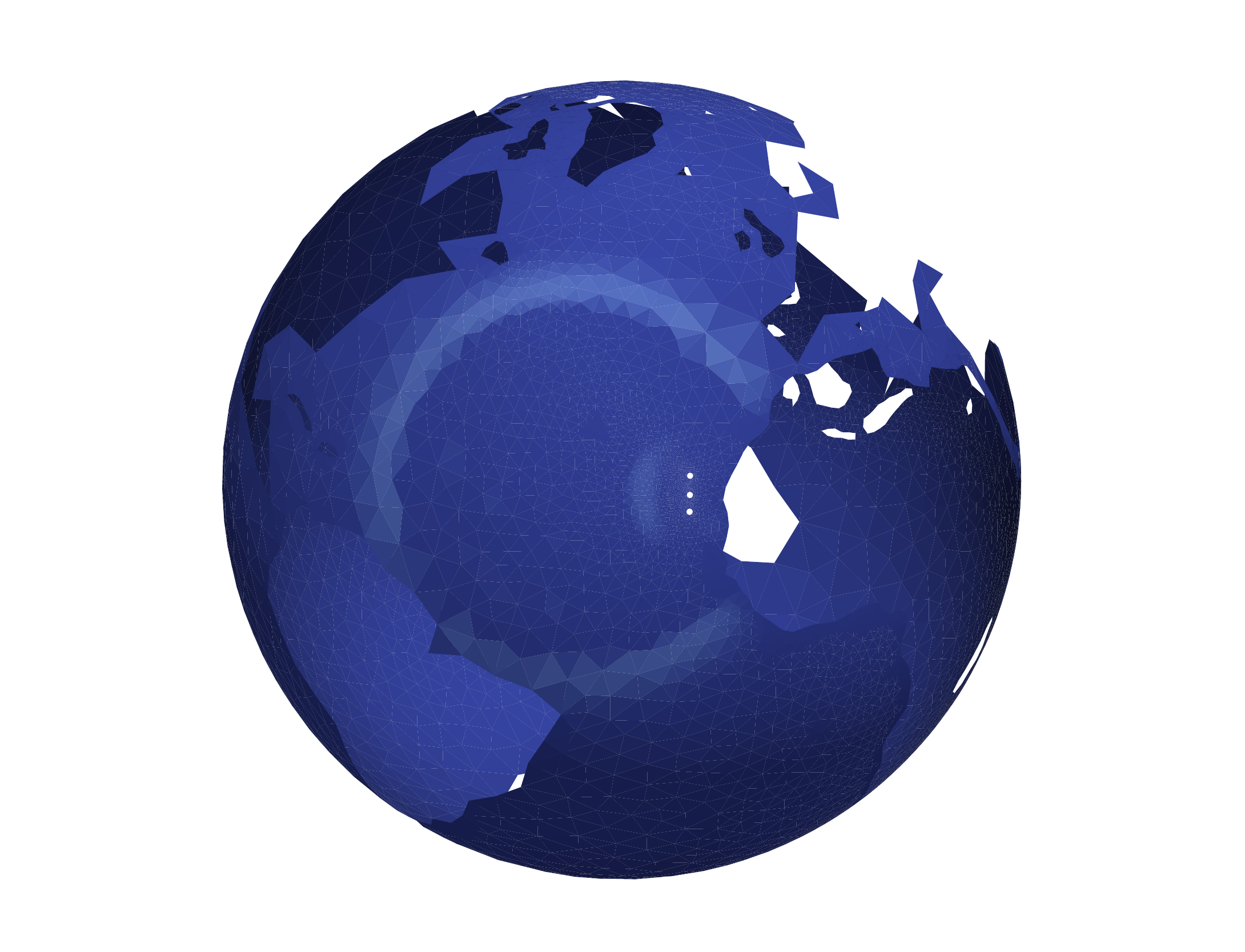}};
	\end{tikzpicture}
	\caption{$t=1$}
\end{subfigure}
\hfill
	\begin{subfigure}{0.2\textwidth}
	\centering
		\begin{tikzpicture}
	\node[anchor=south west,inner sep=0] (1) 
	{\includegraphics[scale=0.11]{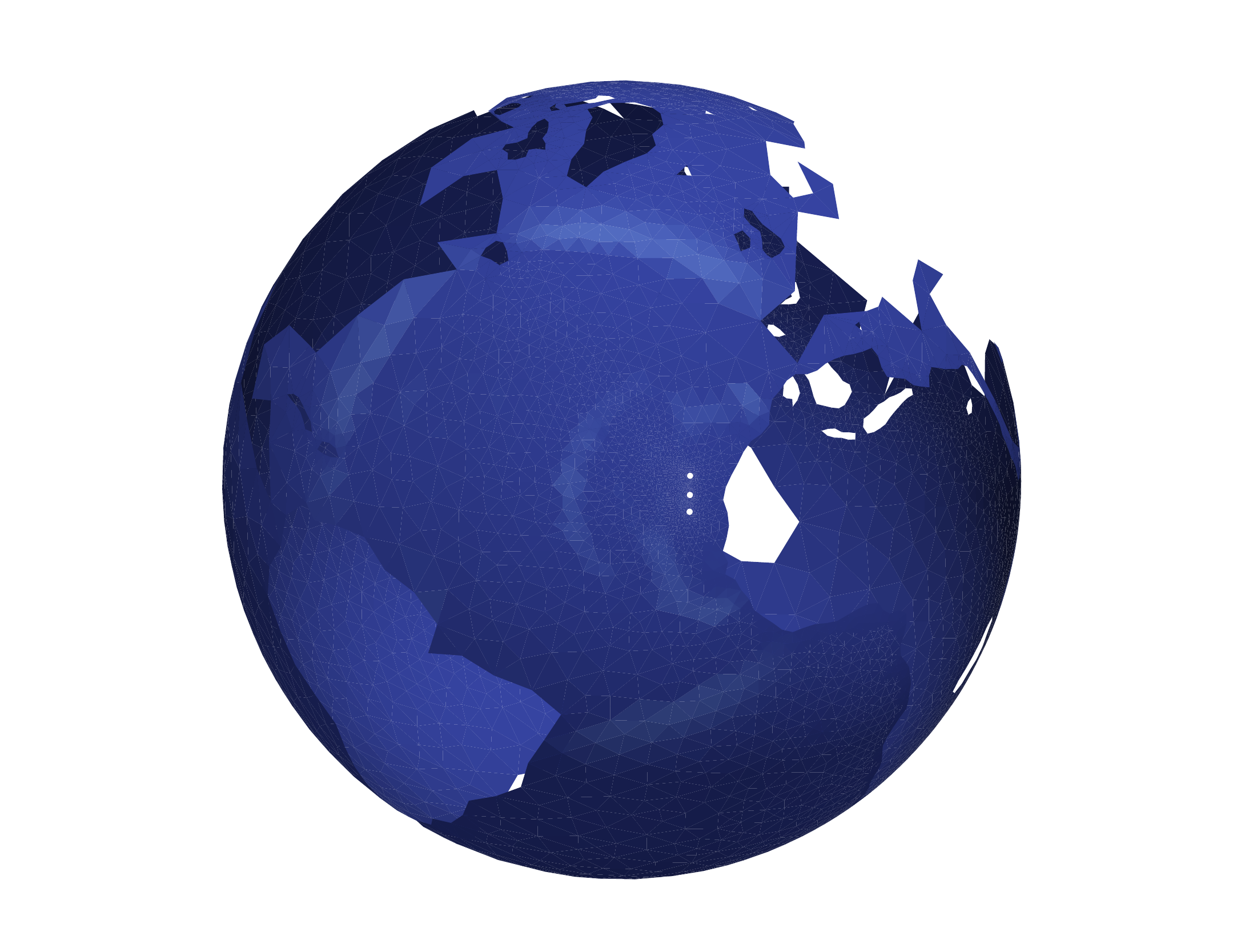}};
	\end{tikzpicture}
	\caption{$t=1.5$}
\end{subfigure}
	\caption[Ex.3 Wave Propagation]{Visualization of a Wave Described by Height, Travelling Towards the Continents for Initial Obstacle.}
\label{fig:WavePropagationSphere}
\end{figure}

 For performing shape optimization we remark for completeness that updating the finite element mesh in each iteration is done via the solution $\vec{W}:\Om\rightarrow\R^3$ of the linear elasticity equation, where we again enforce a tangential solution and hence solve
		\begin{equation}
\begin{aligned}
\int_\Om\left[\sigma(\vec{W}):\eps(\Vv)-l\vec{k}\cdot \Vv+ \vec{W}\cdot\vec{k}\gamma\right]\diff x&=DJ(\Om)[\Vv]& \hspace{1cm} \quad\\
\frac{\partial \vec{W}}{\partial \nv}&=0 \quad &\text{on }&& \Ge&&\\
\vec{W}&=0 \quad &\text{on } &&\Ga,\Gb&&
\end{aligned}
\label{Eq:29LinearElasticityMod}
\end{equation}
for unit outward normal $\vec{k}$ to the surface of the manifold, Lagrange multiplier $l\in DG_1$ for all $(\Vv,\gamma)$ such as $\sigma$ and $\eps$ as in (\ref{Eq:29LinearElasticity}).  We would like to highlight that (\ref{Eq:29LinearElasticityMod}) represents an elliptic PDE, that can without further ado being solved directly. However, movements on a manifold would typically call for retractions, e.g. via usage of an exponential mapping \cite[Chapter 4]{Absil2008}. The resulting deformed obstacles can be seen in Figure \ref{fig:DeformedObstacles}.

\begin{figure}[htb!]
	\centering
	\begin{subfigure}{0.4\textwidth}
		\centering
	
\begin{tikzpicture}
\node[anchor=south west,inner sep=0] (1) {
	\includegraphics[width=4.7cm,height=3.0cm]{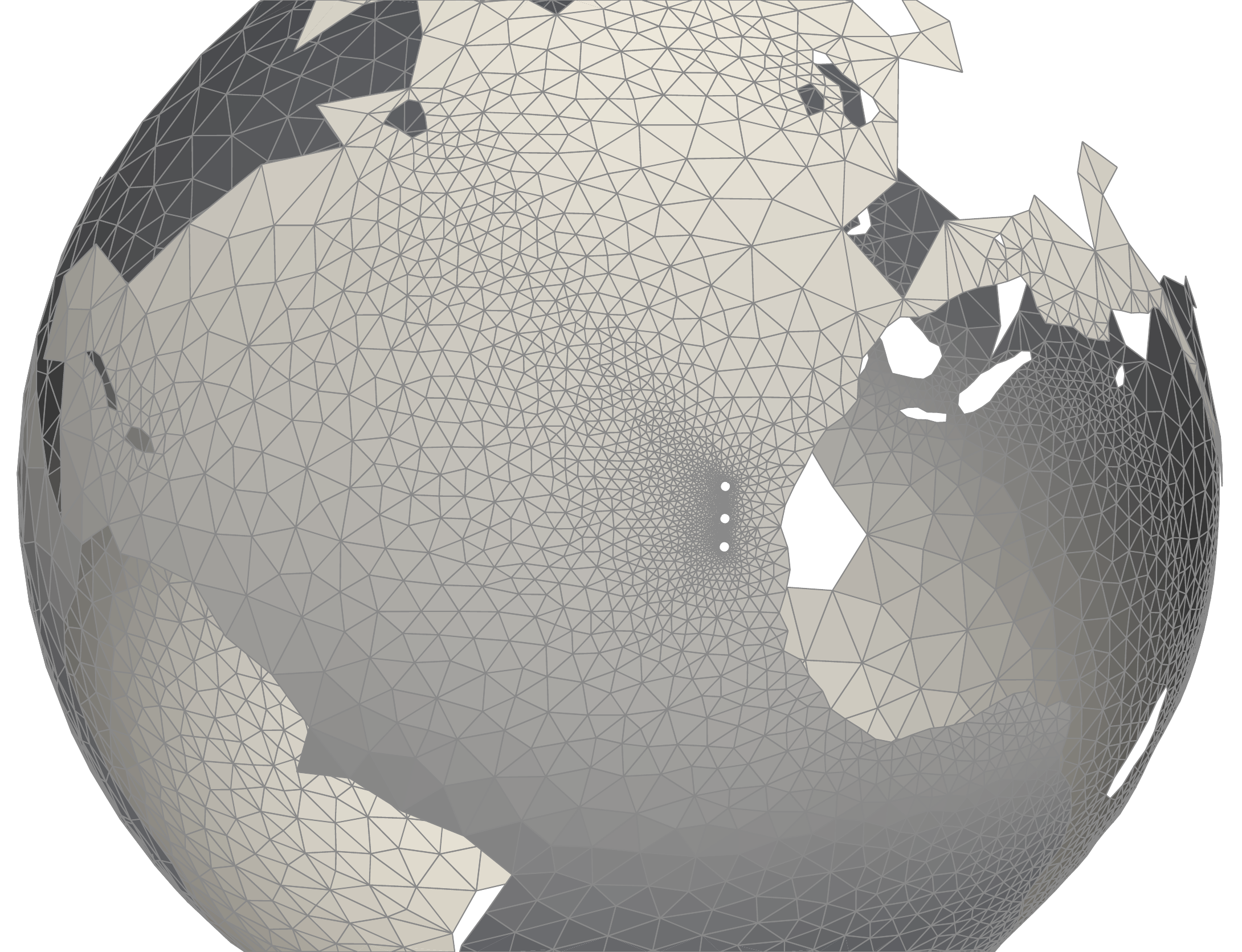}};
\end{tikzpicture}
		\caption{Initial Mesh}
	\end{subfigure}
	\hfill
	\begin{subfigure}{0.4\textwidth}
		\centering
	\begin{tikzpicture}
\node[anchor=south west,inner sep=0] (1)
{\includegraphics[width=4.7cm,height=3.0cm]{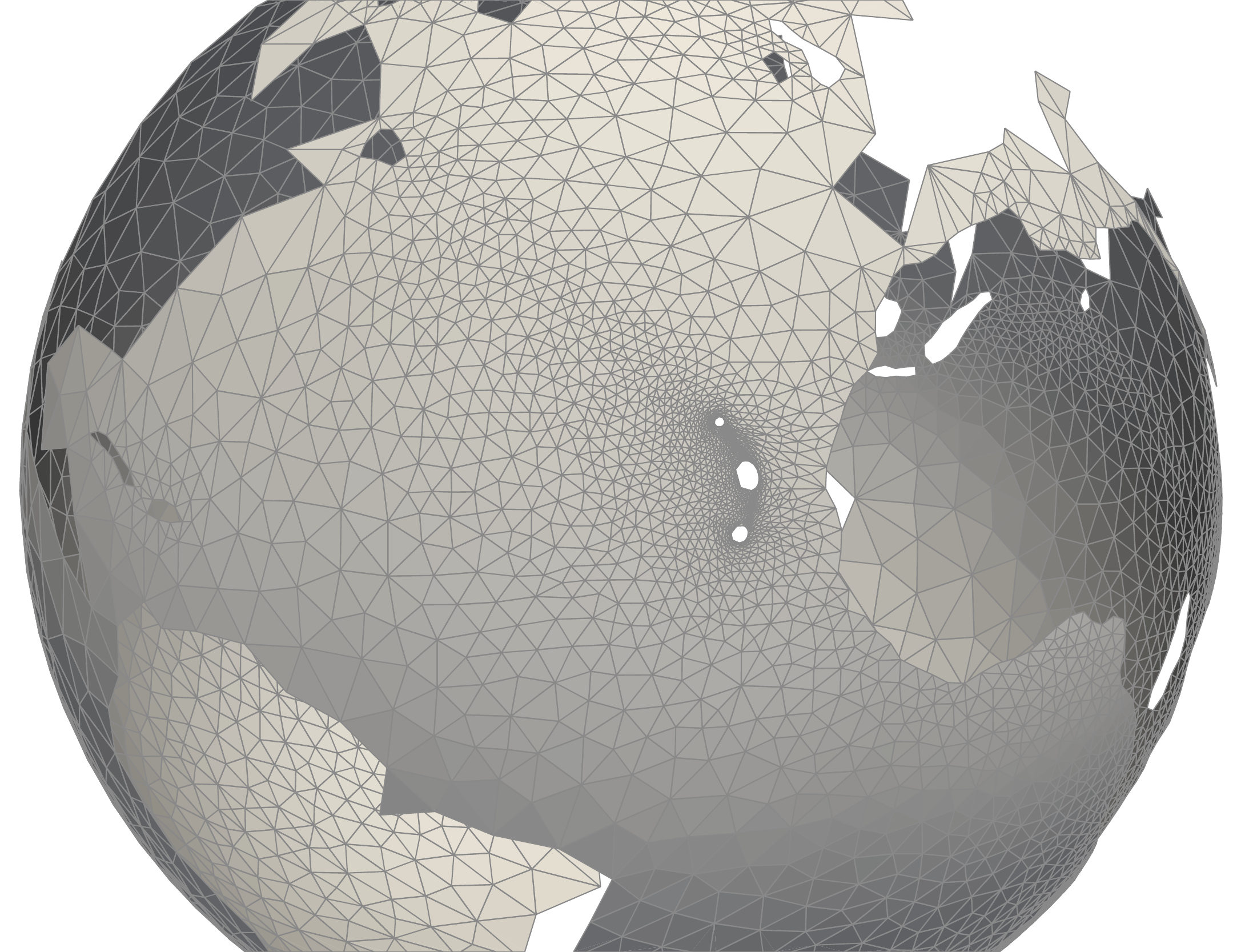}};
\end{tikzpicture}
		\caption{Optimized Mesh}
	\end{subfigure}
	\caption[Ex.3 Initial \& Optimized Mesh \& Obstacle]{Initial and Optimized Mesh and Obstacle}
	\label{fig:DeformedObstacles}
\end{figure}
In Figure \ref{fig:WorldMeshObj} we once more observe convergence of the objective function.
\begin{figure}[!htbp]
	\centering
\scalebox{0.48}{\input{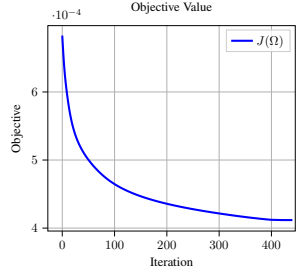}}
	\caption[Ex.3 Objective]{Objective for World Mesh}
	\label{fig:WorldMeshObj}
\end{figure}

Lastly, we would like to point out that the obtained results are only offering a simplistic analysis to protect the shore of Africa $\Ga$, that can be used as a first feasibility study. For a more comprehensive discussion one would need to adapt the model to non-shallow flows, simulate a non-flat seabed and take care on the wetting-drying phenomenon (cf. e.g. to \cite{Karna2011}). On coastal boundaries $\Ga,\Gb$ and $\Ge$ more accurate solutions would be obtained by replacing rigid boundary conditions by partially absorbing boundary conditions. Finally, an extension of $\Ga$ to all shores where various waves are produced with multiple obstacles placed before several shorelines, that are all restricted in volume, could lead to more sophisticated conclusions.

\section{Conclusion}
We have derived the time-dependent continuous adjoint and shape derivative of the SWE in volume form. The results were tested on a simplistic sample mesh for a linear and Gaussian seabed, as well as on more realistic meshes, picturing the Langue de Barbarie coastal section and a world simulation. The optimized shape strongly orients itself to the wave direction and to the mesh region that is to be protected. The results can be easily adjusted for arbitrary meshes, objective functions and different wave properties driven by initial and boundary conditions. However, the obtained obstacles are often too large for practical implementations, hence we admit that this work can only serve as a first feasibility study.


\section*{Acknowledgement}
This work has been supported by the Deutsche
Forschungsgemeinschaft within the Priority program SPP 1962 "Non-smooth and Complementarity-based Distributed Parameter Systems: Simulation and Hierarchical Optimization". The authors would like to thank Diaraf Seck (Université Cheikh Anta Diop, Dakar, Senegal) and Mame Gor Ngom (Université Cheikh Anta Diop, Dakar, Senegal) for helpful and interesting discussions within the project Shape Optimization Mitigating Coastal Erosion (SOMICE). 

\bibliographystyle{unsrt}
\bibliography{bibliography}  
\appendix

\section{Derivation of Adjoint Equations}\label{app:derivadj}
\begin{proof}
	We need to rewrite the weak form (\ref{Eq:17aweakSWE}) as 
	\begin{align*}
	a(H,\Qv,p,\rv)=&\int_0^T\int_\Om-\frac{\partial p}{\partial t}H\diff x\diff t+\int_\Om \left[H(x,T)p(x,T)-H_0p(x,0)\right]\diff x\\
	+&\int_0^T\int_\Om-\Qv\cdot\nab p\diff x\diff t+\int_0^T\int_\Gamma p\Qv\cdot \nv\diff s\diff t\\
	+&\int_0^T\int_\Om-(H+z)\nab\cdot(\mu_v\nab p)\diff x\diff t\\
	+&\int_0^T\int_\Gamma\left[\mu_v(H+z)\nab p\cdot \nv-p\mu_v\nab (H+z)\cdot \nv\right]\diff s\diff t\\
	+&\int_0^T\int_\Om-\frac{\partial \rv}{\partial t}\cdot\Qv\diff x\diff t+\int_\Om \left[\Qv(x,T)\cdot\rv(x,T)-\Qv_0\cdot\rv(x,0)\right]\diff x\\
	+&\int_0^T\int_\Om-\frac{\Qv}{H}\cdot\nab\rv\cdot\Qv\diff x\diff t+\int_0^T\int_\Gamma  \frac{\Qv}{H}\cdot\rv\Qv\cdot \nv\diff s\diff t\\
	+&\int_0^T\int_\Om-\frac{1}{2}gH^2\nab\cdot\rv\diff x\diff t+\int_0^T\int_\Gamma \frac{1}{2}gH^2\rv\cdot \nv\diff s\diff t\\
	+&\int_0^T\int_\Om-\Qv\cdot
	\nab\cdot(G(\mu_f)\nab \rv)\diff x\diff t+\int_0^T\int_\Om gH\nab z \cdot\rv\diff x\diff t\\
	+&\int_0^T\int_\Gamma\left[G(\mu_f)\Qv\cdot\nab \rv\cdot \nv-\rv\cdot G(\mu_f)\nab \Qv\cdot \nv\right]\diff s\diff t\text{.}
	\end{align*}
	Inserting Boundary Conditions leads to
	\begin{align*}
	a(H,\Qv,p,\rv)=&\int_0^T\int_\Om-\frac{\partial p}{\partial t}H\diff x\diff t+\int_\Om \left[H(x,T)p(x,T)-H_0p(x,0)\right]\diff x\\
	-&\int_0^T\int_\Om\Qv\cdot\nab p\diff x\diff t+\int_{0}^T\int_{\Gd}p\Qv\cdot\nv\diff s\diff t\\
	-&\int_0^T\int_\Om\frac{1}{2}gH^2\nab\cdot\rv\diff x\diff t-\int_0^T\int_\Om(H+z)\nab\cdot(\mu_v\nab p)\diff x\diff t\\
	+&\int_0^T\int_{\Gd}-p\mu_v\nab (H_1+z)\cdot \nv\diff s\diff t\\	
	+&\int_0^T\int_{\Ga,\Ge}\mu_v(H+z)\nab p\cdot \nv\diff s\diff t+\int_0^T\int_{\Gd}\mu_vH_1\nab p\cdot \nv\diff s\diff t\\
	-&\int_0^T\int_\Om\frac{\partial \rv}{\partial t}\cdot\Qv\diff x\diff t+\int_\Om \left[\Qv(x,T)\cdot\rv(x,T)-\Qv_0\cdot\rv(x,0)\right]\diff x\\
	-&\int_0^T\int_\Om\frac{\Qv}{H}\cdot\nab\rv\cdot\Qv\diff x\diff t+\int_0^T\int_{\Gd} \frac{\Qv}{H_1}\cdot\rv\Qv\cdot \nv\diff s\diff t\\
	+&\int_0^T\int_{\Ga,\Ge} \frac{1}{2}gH^2\rv\cdot \nv\diff s\diff t+\int_0^T\int_{\Gd} \frac{1}{2}gH_1^2\rv\cdot \nv\diff s\diff t\\
	-&\int_0^T\int_\Om \Qv\cdot
	\nab\cdot(G(\mu_f)\nab \rv)\diff x\diff t+\int_0^T\int_{\Ga,\Gd,\Ge}G(\mu_f)\Qv\nab \rv\cdot\nv\diff s\diff t\\
	+&\int_0^T\int_\Om gH\nab z\cdot\rv\diff x\diff t\text{.}
	\end{align*}
	Differentiating for the state variable $H$ leads to
	\begin{align*}
	\frac{\partial a(H,\Qv,p,\rv)}{\partial H}=&\int_0^T\int_\Om-\frac{\partial p}{\partial t}\diff x\diff t+\int_\Om p(x,T)\diff x\\
	+&\int_0^T\int_\Om-\nab\cdot(\mu_v\nab p)\diff x\diff t+\int_0^T\int_{\Ga,\Ge}\left[\mu_v\nab p\cdot \nv\right]\diff s\diff t\\
	+&\int_0^T\int_\Om\frac{\Qv}{H^2}\cdot\nab\rv\cdot\Qv\diff x\diff t\\
	+&\int_0^T\int_\Om-gH\nab\cdot\rv\diff x\diff t+\int_0^T\int_{\Ga,\Ge} gH\rv\cdot \nv\diff s\diff t\\
	+&\int_0^T\int_\Om g\nab z\cdot\rv\diff x\diff t
	\end{align*}
	and for $\Qv$ to
	\begin{align*}
	\frac{\partial a(H,\Qv,p,\rv)}{\partial \Qv}=		&\int_0^T\int_\Om-\frac{\partial \rv}{\partial t}\diff x\diff t+\int_\Om \rv(x,T)\diff x\\
	-&\int_0^T\int_\Om\nab p\diff x\diff t+\int_0^T\int_{\Gd}p\nv\diff s\diff t\\
	-&\int_0^T\int_\Om\frac{1}{H}(\nab\rv)^T\Qv-\frac{1}{H}(\Qv\cdot\nab)\rv\Qv\diff x\diff t\\
	+&\int_0^T\int_{\Gd} \frac{1}{H_1}(\Qv\cdot \nv)\rv\diff s\diff t+\int_0^T\int_{\Gd} \frac{1}{H_1}(\Qv\rv) \cdot \nv\diff s\diff t\\
	+&\int_0^T\int_\Om-\nab\cdot(G(\mu_f)\nab\rv)\diff x\diff t+\int_0^T\int_{\Ga,\Gd,\Ge}G(\mu_f)\nab\rv\nv\diff s\diff t\text{.}
	\end{align*}
	Now if $\frac{\partial a(H,\Qv,p,\rv)}{\partial U}=-\frac{\partial J_{1,2}}{\partial U}$ then $\frac{\partial\mathcal{L}}{\partial U}=0$ is fulfilled. From this we get the adjoint in strong form with boundary and terminal conditions (\ref{Eq:19Adjoint})-(\ref{Eq:20AdjointBC}).
\end{proof}
\newpage
\section{Derivation of Shape Derivative}\label{app:derivsha}
\begin{proof}
	We regard the Lagrangian (\ref{Eq:16LagSWE}). As in \cite{Schulz2014b}, the theorem of Correa and Seger \cite{Correa1985} is applied on the right hand side of 
	\begin{align}
	J_{1,2}(\Om)=\min_{U}\max_{P}\Lag(\Om,U,P).
	\label{Eq:191Saddle}
	\end{align}
	The assumptions of this theorem can be verified as in \cite{Delfour2011}. We now apply the rule \eqref{Eq:9MatDer2} for differentiating domain integrals, alongside with boundary conditions
	\allowdisplaybreaks
	\begin{align*}
	d\mathcal{L}&(\Om,U,P)=\\
	&=\lim_{\eps\rightarrow 0^+}\frac{\mathcal{L}(\Om_\eps;U,P)-\mathcal{L}(\Om;U,P)}{\eps}\\
	&=\frac{d^+}{d\eps}\restr{\mathcal{L}(\Om_\eps,U,P)}{\eps=0}=\frac{d^+}{d\eps}\restr{\mathcal{L}(\Om_\eps,H,\Qv,p,\rv)}{\eps=0}\\
	&=\int_\Om\Big[\int_0^T-D_m\left(\frac{\partial p}{\partial t}H\right)\diff t+ D_m\left(H(x,T)p(x,T)-H_0p(x,0)\right)\\
	&-\int_0^TD_m\left(\frac{\partial \rv}{\partial t}\cdot\Qv\right)\diff t+ D_m\left(\Qv(x,T)\cdot\rv(x,T)-\Qv_0\cdot\rv(x,0)\right)\\
	&+\int_0^T D_m\left(\nab\cdot\Qv p\right)\diff t+\int_0^TD_m\left(\mu_v\nab (H+z)\cdot\nab p\right)\diff t\\
	&+\int_{0}^TD_m\left(\nab\cdot\left(\frac{\Qv}{H}\otimes \Qv\right)\cdot\rv\right)\diff t+
	\int_0^T+D_m\left(\frac{1}{2}g\nab H^2\cdot\rv\right)\diff t\\
	&+\int_0^TD_m\left(G(\mu_f)\nab \Qv:\nab\rv\right)\diff t+\int_0^TD_m\left(gH\nab z\cdot\rv\right)\diff t
	\\
	&+\Div(\Vv)\Big(\int_0^T-\frac{\partial p}{\partial t}H\diff t+H(x,T)p(x,T)-H_0p(x,0)\\
	&+\int_0^T-\frac{\partial \rv}{\partial t}\cdot\Qv\diff t+ \Qv(x,T)\cdot\rv(x,T)-\Qv_0\cdot\rv(x,0)+\int_0^T \nab\cdot\Qv p\diff t\\
	&+\int_0^T\mu_v\nab (H+z)\cdot\nab p\diff t+\int_{0}^T\nab\cdot\left(\frac{\Qv}{H}\otimes \Qv\right)\cdot\rv\diff t\\
	&+\int_0^T\frac{1}{2}g\nab H^2\cdot\rv\diff t+\int_0^TG(\mu_f)\nab \Qv:\nab\rv+\int_0^TgH\nab z\cdot\rv\diff t\Big)\Big]\diff x\\
	&+\int_{\Ga}\Big[\int_{\tilde{T}}D_m\left(\nu_1E\sigma_\alpha(H-H_{\text{cr}})\right)\diff t+\int_0^TD_m\left(\frac{\nu_2}{2}||\Qv||_2^2\right)\diff t\\
	&+\Div_{\Ga}(\Vv)\left(\int_{\tilde{T}}\nu_1E\sigma_\alpha(H-H_{\text{cr}})\diff t+\int_0^T\frac{\nu_2}{2}||\Qv||_2^2\diff t\right)\Big]\diff s
	\\
	&+\int_{\Gd}\Big[\int_0^T-D_m\left(\mu_v\nab (H_1+z) \cdot \nv p\diff t\right)\\
	&+\Div_{\Gd}(\Vv)\Big(\int_0^T-\mu_v\nab (H_1+z)\cdot \nv p\diff t\Big)\Big]\diff s\text{,}
	\end{align*}
	where $\Div_{\Gamma}\Vv=\Div\Vv-\nv\cdot(\nab\Vv)\nv$ is the tangential divergence of the vector field $\Vv$.
	Now the product rule (\ref{Eq:10MatProdR}) yields
	\allowdisplaybreaks
	\begin{align*}
	\quad\quad\quad=&\int_\Om\Big[\int_0^T-D_m\left(\frac{\partial p}{\partial t}\right)H-\frac{\partial p}{\partial t}\dot{H}\diff t \\
	+&\dot{H}(x,T)p(x,T)+H(x,T)\dot{p}(x,T)-H_0\dot{p}(x,0)\\
	+&\int_0^TD_m\left(\frac{\partial \rv}{\partial t}\right)\cdot\Qv-\frac{\partial \rv}{\partial t}\cdot\dot{\Qv}\diff t+ \dot{\Qv}(x,T)\cdot\rv(x,T)\\
	+&\Qv(x,T)\cdot\dot{\rv}(x,T)-\Qv_0\cdot\dot{\rv}(x,0)+
	\int_0^T \dot{p}\cdot\nab\cdot \Qv+p D_m(\nab\cdot \Qv)\diff t\\
	+&\int_0^T\left(\mu_vD_m(\nab (H+z))\cdot\nab p+\mu_v\nab (H+z)\cdot D_m(\nab p)\right)\diff t\\
	-&\int_{0}^TD_m\left(\nab\cdot\left(\frac{\Qv}{H}\otimes \Qv\right)\right)\cdot\rv\diff t+\int_{0}^T\nab\cdot\left(\frac{\Qv}{H}\otimes \Qv\right)\cdot D_m\left(\rv\right)\diff t\\
	+& \int_0^T\left(\frac{1}{2}gD_m(\nab H^2)\cdot\rv+\frac{1}{2}g\nab H^2\cdot D_m(\rv)\right)\diff t
	\\
	+&\int_0^T\left(D_m\left(G(\mu_f)\nab \Qv\right):\nab\rv+G(\mu_f)\nab \Qv:D_m\left(\nab\rv\right)\right)\diff t\\
	+&\int_0^Tg\dot{H}\nab z\cdot\rv\diff t+\int_0^TgHD_m(\nab z)\cdot\rv\diff t+\int_0^TgH\nab z\cdot\dot{\rv}\diff t\\
	+&\quad \Div(\Vv)\Big(\int_0^T-\frac{\partial p}{\partial t}H\diff t+H(x,T)p(x,T)-H_0p(x,0)\\
	+&\int_0^T-\frac{\partial \rv}{\partial t}\cdot\Qv\diff t+ \Qv(x,T)\cdot\rv(x,T)-\Qv_0\cdot\rv(x,0)+\int_0^T p\nab\cdot\Qv\diff t\\
	+&\int_0^T\mu_v\nab (H+z)\cdot\nab p\diff t+\int_{0}^T\nab\cdot\left(\frac{\Qv}{H}\otimes \Qv\right)\cdot\rv\diff t\\
	+&\int_0^T+\frac{1}{2}g\nab H^2\cdot\rv\diff t+\int_0^TG(\mu_f)\nab \Qv:\nab\rv+\int_0^TgH\nab z\cdot\rv\diff t\Big)\Big]\diff x\\
	+&\int_{\Ga}\Big[\int_{\tilde{T}}\nu_1\left(\frac{1}{4}g\rho
	H\sigma_{\alpha,H_{\text{cr}}}(H)+E\sigma_{\alpha,H_{\text{cr}}}(H)(1-\sigma_{\alpha,H_{\text{cr}}}(H))\right)\dot{H}\diff t\\
	+&\int_0^T\nu_2\Qv\cdot\dot{\Qv}\diff t\\
	+&\quad \Div_{\Ga}(\Vv)\left(\int_{\tilde{T}}\nu_1E\sigma_{\alpha,H_{\text{cr}}}(H)\diff t+\int_0^T\frac{\nu_2}{2}||\Qv||_2^2\diff t\right)\Big]\diff s\\
	+&\int_{\Gd}\Big[\int_0^T-\mu_v\nab (H_1+z)\cdot \nv\dot{p}\diff t+ \Div_{\Gd}(\Vv)\Big(\int_0^T-\mu_v\nab (H_1+z)\cdot \nv p\diff t\Big)\Big]\diff s\text{.}
	\end{align*}
	The non-commuting of the material derivative (\ref{Eq:11MatGradR}), (\ref{Eq:11MatGradRvec}) and (\ref{Eq:12MatGradProdR}) such as integration by parts, regrouping and the fact that the sediment moves along with the deformation leads to
	\allowdisplaybreaks
	\begin{align*}
	\quad\quad=&\int_{\Ga}\Big[\int_{\tilde{T}}\left(\frac{1}{4}g\rho
	H\sigma_{\alpha,H_{\text{cr}}}(H)+E\sigma_{\alpha,H_{\text{cr}}}(H)(1-\sigma_{\alpha,H_{\text{cr}}}(H))\right)\dot{H}\diff t\\
	+&\int_0^T\nu_2\Qv\cdot\dot{\Qv}\diff t\Big]\diff s\\
	+&\int_\Om\Big[\int_0^T\left(-\frac{\partial p}{\partial t}+\frac{1}{H^2}(\Qv\cdot\nab)\rv\cdot\Qv-gH(\nab\cdot\rv) -\nab\cdot(\mu_v\nab p)+g\nab z\cdot \rv\right)\dot{H}\\
	+&\left(-\frac{\partial\rv}{\partial t}-\nab p-\frac{1}{H}(\Qv\cdot\nab)\rv-\frac{1}{H}(\nab\rv)^T\Qv-(\nab\cdot(G(\mu_f)\nab \rv))\right)\cdot\dot{\Qv}\\	
	+&\left(\frac{\partial H}{\partial t}+\nab\cdot\left(\Qv-\mu_v\nab (H+z)\right)\right)\dot{p}\\
	+&\left(\frac{\partial \Qv}{\partial t}+\nab\cdot\left(\frac{\Qv}{H}\otimes \Qv+\frac{1}{2}gH^2\mathbf{I}_2-G(\mu_f)\nab\Qv\right)+gH\nab z\right)\cdot\dot{\rv}\diff t\Big]\diff x\\
	+&\int_{\Om}\int_{0}^{T}\Big[-(\nab \Vv)^T:\nab \Qv p - (\nab \Vv)^T:\nab \Qv\frac{\Qv}{H}\cdot \rv- (\nab \Vv\Qv\cdot\nab)\frac{\Qv}{H}\cdot \rv \\ 
	-& gH(\nab \Vv)^T\nab H \cdot \rv-\mu_v\nab (H+z)^T(\nab \Vv +\nab \Vv^T)\nab p  \\
	-&G(\mu_f)\nab \Qv \nab \Vv:\nab \rv -G(\mu_f)\nab\Qv\nab \Vv^T:\nab\rv  \\
	-&gH\nab \Vv^T\nab z\cdot\rv+\Div(\Vv)\Big\{\frac{\partial H}{\partial t}p+\nab\cdot \Qv p+\frac{\partial \Qv}{\partial t}\cdot \rv\\
	+& (\Qv\cdot \nab)\frac{\Qv}{H}\cdot \rv +\nab\cdot \Qv\frac{\Qv}{H}\cdot \rv +  \frac{1}{2}g\nab H^2\cdot \rv +gH\nab z\cdot\rv\\
	+&\mu_v\nab (H+z)\cdot\nab p +(G(\mu_f)\nab \Qv): \nab \rv \Big\}\Big]\diff x\diff t\\
	+&\int_{\Ga}\Div_{\Ga}(\Vv)\Big[\int_{\tilde{T}}\nu_1E\sigma_{\alpha,H_{\text{cr}}}(H)\diff t+\int_0^T\frac{\nu_2}{2}||\Qv||_2^2\diff t \Big]\diff s\\
	+& \int_{\Gd}\Div_{\Gd}(\Vv)\Big[\int_0^T-\mu_v\nab (H_1+z)\cdot \nv p\diff t\Big]\diff s\text{.}
	\end{align*}
	Since outer boundaries are not variable, in general the deformation field $\Vv$ vanishes in small neighbourhoods around $\Ga, \Gd$ and the material derivative is zero, hence the boundary integrals vanish. In addition, evaluating the Lagrangian in its saddle point, the 
	first integrals vanish such that we obtain the shape derivative in its final form.
\end{proof}

\end{document}